\documentclass[12pt,twoside,a4paper]{amsart}
\usepackage{amssymb,amscd,amsxtra,calc,mathrsfs}
\usepackage{amsmath}
\usepackage{xypic}

\title[K-stability of asymptotically log del Pezzo surfaces]{On the uniform K-stability 
for some asymptotically log del Pezzo surfaces}
\author{Kento Fujita} 
\date{\today}
\address{Department of Mathematics, Graduate School of Science, Osaka University, 
Toyonaka, Osaka 560-0043, Japan}
\email{fujita@math.sci.osaka-u.ac.jp}

\newcommand{\pr}{\mathbb{P}}

\newcommand{\Z}{\mathbb{Z}}
\newcommand{\Q}{\mathbb{Q}}
\newcommand{\R}{\mathbb{R}}
\newcommand{\C}{\mathbb{C}}

\newcommand{\A}{\mathbb{A}}
\newcommand{\G}{\mathbb{G}}

\newcommand{\Hom}{\operatorname{Hom}}

\newcommand{\Aut}{\operatorname{Aut}}

\newcommand{\ord}{\operatorname{ord}}

\newcommand{\vol}{\operatorname{vol}}

\newcommand{\PGL}{\operatorname{PGL}}

\newcommand{\diag}{\operatorname{diag}}

\newcommand{\coeff}{\operatorname{coeff}}

\newcommand{\sC}{\mathcal{C}}

\newcommand{\sO}{\mathcal{O}}

\newtheorem{thm}{Theorem}[section]
\newtheorem{lemma}[thm]{Lemma}
\newtheorem{proposition}[thm]{Proposition}
\newtheorem{corollary}[thm]{Corollary}

\theoremstyle{definition}
\newtheorem{definition}[thm]{Definition}
\newtheorem{remark}[thm]{Remark}

\newtheorem*{ack}{Acknowledgments}

\begin{document}

\maketitle 

\begin{abstract}
Motivated by the problem for the existence of K\"ahler-Einstein edge metrics, 
Cheltsov and Rubinstein conjectured the K-polystability of asymptotically 
log Fano varieties with small cone angles when the anti-log-canonical divisors are not big. 
Cheltsov, Rubinstein and Zhang proved it affirmatively in dimension $2$ with irreducible 
boundaries except for the type $(\operatorname{I.9B.}n)$ with $1\leq n\leq 6$. Unfortunately, recently, 
Fujita, Liu, S\"u\ss, Zhang and Zhuang showed the non-K-polystability for some 
members of type $(\operatorname{I.9B.}1)$ and for some members of type 
$(\operatorname{I.9B.}2)$. In this article, 
we show that Cheltsov--Rubinstein's problem is true for all of the remaining cases. 
More precisely, we explicitly compute the delta-invariant for asymptotically 
log del Pezzo surfaces of type $(\operatorname{I.9B.}n)$ for all $n\geq 1$ 
with small cone angles. 
As a consequence, we finish Cheltsov--Rubinstein's problem in dimension $2$
with irreducible boundaries. 
\end{abstract}

\setcounter{tocdepth}{1}
\tableofcontents

\section{Introduction}\label{intro_section}

Consider a $d$-dimensional smooth projective variety $X$ 
together with a smooth prime divisor $C$ on $X$. The pair $(X, C)$ is said to be 
an \emph{asymptotically log Fano variety} (with an irreducible boundary) if 
$-(K_X+(1-\beta)C)$ is ample for any sufficiently small $0<\beta\ll 1$. The notion was 
introduced in \cite{CR13}, motivated by the existence of K\"ahler-Einstein edge metrics 
with cone angle $2\pi\beta$. From the base-point-free theorem, the divisor 
$-(K_X+C)$ is semiample and induces the contraction morphism $\eta\colon X\to Z$ 
with $\eta_*\sO_X=\sO_Z$. We call the morphism the \emph{anti-log-canonical 
morphism}. In \cite{CR13}, the authors conjectured the following: 
\begin{enumerate}
\renewcommand{\theenumi}{\Alph{enumi}}
\renewcommand{\labelenumi}{(\theenumi)}
\item\label{intro_conj1}
If $\dim Z=d$, then $(X, C)$ does not admit K\"ahler-Einstein edge metrics with cone 
angle $2\pi\beta$ for any $0<\beta\ll 1$. 
\item\label{intro_conj2}
If $\dim Z<d$, then $(X, C)$ might admits K\"ahler-Einstein edge metrics with cone 
angle $2\pi\beta$ for any $0<\beta\ll 1$. 
\end{enumerate}
In fact, the assertion \eqref{intro_conj1} is proved affirmatively if $d=2$ in \cite{CR15} 
and if $d$ is arbitrary in \cite{logKst}. The assertion \eqref{intro_conj2} is proved 
affirmatively if $\dim Z=0$ (i.e., $C\sim -K_X$) in \cite{JMR} and if $d=2$ except for
few possibilities in \cite{CR13, CRZ}. The results \cite{CR13, CRZ} based on the 
classification of $2$-dimensional asymptotically log Fano varieties (called 
\emph{asymptotically log del Pezzo surfaces}). 
We recall the classification result in \cite{CR13}. 

\begin{thm}[{\cite[Theorems 1.4 and 2.1]{CR13}}]\label{CR_thm}
Let $(X, C)$ be an asymptotically log del Pezzo surface with $\dim Z=1$, where 
$\eta\colon X\to Z$ is the anti-log-canonical morphism. Then $(X, C)$ is isomorphic 
to one of the following types: 
\begin{description}
\item[{$(\operatorname{I.3A})$}]
$X=\pr_{\pr^1}\left(\sO\oplus\sO(1)\right)$, $C$ is a smooth member in 
$|2\xi|$, where $\xi$ is a tautological line bundle of 
$\pr_{\pr^1}\left(\sO\oplus\sO(1)\right)/\pr^1$. 
\item[{$(\operatorname{I.4B})$}]
$X=\pr^1\times\pr^1$, $C$ is a smooth member in $|\sO_{\pr^1\times\pr^1}(1,2)|$. 
\item[{$(\operatorname{I.9B.}n)$}\,\, $(n\geq 1)$]
$X$ is the blowup of $\pr^1\times\pr^1$ at $n$ numbers of distinct points lying on 
a smooth curve $C'\in|\sO_{\pr^1\times\pr^1}(1,2)|$ such that the composition 
$X\to\pr^1\times\pr^1\xrightarrow{p}\pr^1$ is a conic bundle, i.e., 
$-K_X$ is ample over $\pr^1$, $C$ is the strict transform of $C'$, where 
$p\colon \pr^1\times\pr^1\to\pr^1$ is the first projection morphism. 
\end{description}
\end{thm}

By the works \cite[Propositions 7.4 and 7.5]{CR13} and \cite[Theorem 1.3]{CRZ}, 
the assertion \eqref{intro_conj2} is turned out to be true except for 
the types $(\operatorname{I.9B.}n)$ with $1\leq n\leq 6$. The strategy in 
\cite{CRZ} is to evaluate the \emph{delta-invariant} $\delta(X, (1-\beta)C)$
introduced in \cite{FO, BJ} 
for $2$-dimensional log Fano pairs $(X, (1-\beta)C)$. 
(We will recall the definition of the 
delta-invariant in \S \ref{prelim_section}.) In fact, it has been shown in \cite{BJ} that 
the \emph{uniform K-stability} (resp., the \emph{K-semistability}) of $(X, (1-\beta)C)$ 
is equivalent to the condition $\delta(X, (1-\beta)C)>1$ (resp., 
$\delta(X, (1-\beta)C)\geq 1$). Moreover, by the works \cite{tian, don, stoppa, 
CDS1, CDS2, CDS3, tian2, TW, LTW} and references therein, the 
\emph{K-polystability} of $(X, (1-\beta)C)$ is known to be equivalent to the existence of 
K\"ahler-Einstein edge metrics with cone angle $2\pi\beta$, where the K-polystability 
is weaker than the uniform K-stability and is stronger than the K-semistability. 

However, recently, the assertion \eqref{intro_conj2} was turned out to be not true 
in \cite{FLSZZ} even when $d=2$. Let $(X, C)$ be an asymptotically del Pezzo surface
of type $(\operatorname{I.9B.}n)$ and let $\eta\colon X\to \pr^1$ be the 
anti-log-canonical morphism. Obviously, the morphism $\eta$ is equal to 
the composition $X\to\pr^1\times\pr^1\xrightarrow{p}\pr^1$ in the definition of 
$(\operatorname{I.9B.}n)$. Under the above setting, the restriction morphism 
$\eta|_C\colon C\to\pr^1$ is a double cover. Thus there are exactly $2$ numbers of 
ramification points $q_1$, $q_2\in C$ of $\eta|_C$. In \cite{FLSZZ}, the authors got the 
following result: 

\begin{thm}[{\cite[\S 2]{FLSZZ}}]\label{FLSZZ_thm}
\begin{enumerate}
\renewcommand{\theenumi}{\arabic{enumi}}
\renewcommand{\labelenumi}{(\theenumi)}
\item\label{FLSZZ_thm1}
If $n=1$ and $q_1$ lies on the singular fiber of the conic bundle 
$\eta\colon X\to\pr^1$, then we have 
\[
\delta\left(X,(1-\beta)C\right)\leq\frac{3\beta +4}{4\beta +4}
\]
for any $\beta\in(0,1]\cap\Q$. Thus $(X,(1-\beta)C)$ is not K-semistable. 
\item\label{FLSZZ_thm2}
Assume that $n=2$ and $\beta\in(0$, $1]\cap\Q$. 
\begin{enumerate}
\renewcommand{\theenumii}{\roman{enumii}}
\renewcommand{\labelenumii}{\rm{(\theenumii)}}
\item\label{FLSZZ_thm21}
If both $q_1$ and $q_2$ lie on singular fibers of $\eta$, then $(X, C)$ admits 
K\"ahler-Einstein edge metrics with cone angle $2\pi\beta$. 
\item\label{FLSZZ_thm22}
If $q_1$ lies on a singular fiber of $\eta$ and if $q_2$ does not lie on any singular fiber 
of $\eta$, then $(X, (1-\beta)C)$ is not K-polystable but K-semistable. 
\end{enumerate}
\end{enumerate}
\end{thm}

Thanks to Theorem \ref{FLSZZ_thm}, we must modify the assertion 
\eqref{intro_conj2}. However, currently, there is no good substitution. In order to 
consider it, it is important to understand the existence of K\"ahler-Einstein edge 
metrics with small cone angles in dimension $2$ completely. The purpose of this article 
is to compute the delta-invariant of the asymptotically log del Pezzo surfaces 
of type $(\operatorname{I.9B.}n)$ with small cone angles. The following is 
the main theorem in this article (see also Theorems \ref{S7_thm}, \ref{caseI_thm} and  \ref{caseII_thm}): 

\begin{thm}[Main Theorem]\label{mainthm}
Let $(X, C)$ be an asymptotically log del Pezzo surface of type 
$(\operatorname{I.9B.}n)$, let $\eta\colon X\to\pr^1$ be the anti-log-canonical 
morphism, and let $q_1$, $q_2\in C$ be the ramification points of the double 
cover $\eta|_C\colon C\to\pr^1$. Take any $\beta\in(0$, \,$1/(7n))\cap\Q$. 
\begin{enumerate}
\renewcommand{\theenumi}{\arabic{enumi}}
\renewcommand{\labelenumi}{(\theenumi)}
\item\label{mainthm1}
If either $q_1$ or $q_2$ lies on a singular fiber of $\eta$, then we have 
\[
\delta\left(X, (1-\beta)C\right)=\frac{3((4-n)\beta+4)(\beta+1)}{(n^2-9n+20)\beta^2
+(-6n+30)\beta+12}. 
\]
\item\label{mainthm2}
Otherwise, we have
\[
\delta\left(X, (1-\beta)C\right)=\frac{3((4-n)\beta+4)(2\beta+1)}{(n^2-10n+24)\beta^2
+(-6n+36)\beta+12}.
\]
\end{enumerate}
\end{thm}

\begin{remark}\label{main_rmk}
\begin{enumerate}
\renewcommand{\theenumi}{\arabic{enumi}}
\renewcommand{\labelenumi}{(\theenumi)}
\item\label{main_rmk1}
We have the inequality 
\begin{eqnarray*}
\frac{3((4-n)\beta+4)(2\beta+1)}{(n^2-10n+24)\beta^2
+(-6n+36)\beta+12}
>\frac{3((4-n)\beta+4)(\beta+1)}{(n^2-9n+20)\beta^2
+(-6n+30)\beta+12}. 
\end{eqnarray*}
\item\label{main_rmk2}
We have 
\[
\frac{3((4-n)\beta+4)(\beta+1)}{(n^2-9n+20)\beta^2
+(-6n+30)\beta+12}
\begin{cases}
=\frac{3\beta+4}{4\beta+4} & \text{if }n=1, \\
=1 & \text{if }n=2, \\
\in\left(1,\,\frac{3}{2}\right) & \text{if }n\geq 3.
\end{cases} 
\]
Thus Theorem \ref{mainthm} generalizes Theorem \ref{FLSZZ_thm} \eqref{FLSZZ_thm1} 
with small cone angles. 
Moreover, the proof of Theorem \ref{mainthm} gives an alternative proof of 
Theorem \ref{FLSZZ_thm} \eqref{FLSZZ_thm2} with small cone angles. See Remarks 
\ref{61_rmk} and \ref{62_rmk}. 
\item\label{main_rmk3}
We have 
\[
\frac{3((4-n)\beta+4)(2\beta+1)}{(n^2-10n+24)\beta^2
+(-6n+36)\beta+12}\,\,\in\,\,\left(1,\,\,\,\frac{3}{2}\right)
\]
for any $n \geq 1$. 
\end{enumerate}
\end{remark}

Thanks to Theorem \ref{mainthm}, together with the results in \cite{CR13, CR15, JMR, 
FLSZZ}, we finish Cheltsov--Rubinstein's problem in dimension $2$ with irreducible 
boundaries. The strategy to prove Theorem \ref{mainthm} is to analyze plt-type 
prime divisors over $(X, (1-\beta)C)$ combinatorially. The idea is based on 
the work \cite{KpsMaeda}. If $\beta$ is very small, then the possibility of 
plt-type prime divisors over $(X, (1-\beta)C)$ is very restricted. Moreover, 
we can easily reduce the computation to the case $n=1$ in many cases. 

We work over the complex number field $\C$. However, Theorem \ref{mainthm} and 
all arguments from \S \ref{prelim_section} work over any algebraically closed 
field $\Bbbk$ of characteristic zero. For the minimal model program, we refer the 
readers to \cite{KoMo}; for the toric geometry, we refer the readers to \cite{CLS}. 
In this article, we only treat $2$-dimensional toric varieties. We always fix the lattice 
$N:=\Z^{\oplus 2}$ of rank $2$ and its dual lattice $M:=\Hom_\Z(N,\Z)$. 
We set $N_\R:=N\otimes_\Z\R$ and $M_\R:=M\otimes_\Z\R$. For any 
birational map $\phi\colon X\dashrightarrow X'$ 
between normal projective varieties and 
for any $\Q$-Weil divisor $\Delta$ on $X$, the strict transform of $\Delta$ to $X'$ 
(that is, $\phi_*\Delta$) is often denoted by $\Delta^{X'}$ in this article. For a prime 
divisor $E$ on $X$, the coefficient of $\Delta$ at $E$ is denoted by $\coeff_E\Delta$. 

\begin{ack}
The author thanks Ivan Cheltsov, Yuchen Liu, Hendrik S\"u\ss, Kewei Zhang and 
Ziquan Zhuang for discussions. 
This work was supported by JSPS KAKENHI Grant Number 18K13388.
\end{ack}

\section{Preliminaries}\label{prelim_section}

\subsection{The K-stability of log Fano pairs}\label{K_section}

Let us assume that $(X, \Delta)$ is a $d$-dimensional \emph{log Fano pair}, 
that is, $X$ is a 
$d$-dimensional normal projective variety, $\Delta$ is an effective $\Q$-Weil 
divisor such that $-(K_X+\Delta)$ is an ample $\Q$-Cartier $\Q$-divisor and 
the pair $(X, \Delta)$ is a klt pair. Set $L:=-(K_X+\Delta)$. We recall the notions of the 
uniform K-stability, the K-polystability 
and the K-semistability of $(X, \Delta)$ according to the 
interpretation \cite{li, vst, pltK, hyperplane, BX} of the original one \cite{tian, don}. 

\begin{definition}[{\cite{li, vst, pltK, hyperplane}}]\label{vst_dfn}
Let $F$ be a prime divisor over $X$. Then there exists a log resolution 
$\pi\colon\tilde{X}\to X$ of $(X, \Delta)$ such that $F$ can be realized as a prime 
divisor on $\tilde{X}$. 
\begin{enumerate}
\renewcommand{\theenumi}{\arabic{enumi}}
\renewcommand{\labelenumi}{(\theenumi)}
\item\label{vst_dfn1}
We set 
\[
A_{X, \Delta}(F):=1+\coeff_F\left(K_{\tilde{X}}-\pi^*(K_X+\Delta)\right).
\]
\item\label{vst_dfn2}
We set 
\[
S_L(F):=S(F):=\frac{1}{(L^{\cdot d})}\int_0^\infty\vol_{\tilde{X}}\left(\pi^*L-x F\right)dx, 
\]
where $\vol_{\tilde{X}}$ is the volume function \cite{L1, L2}. The function is continuous 
and non-increasing. Moreover, the function is $\sC^1$ whenever
$\vol_{\tilde{X}}\left(\pi^*L-x F\right)>0$ \cite{BFJ}. 
\item\label{vst_dfn3}
The divisor $F$ is said to be \emph{dreamy} over $(X, \Delta)$ if the algebra
\[
\bigoplus_{k,j\in\Z_{\geq 0}}H^0\left(\tilde{X}, k r\pi^*L-j F\right)
\]
is finitely generated over the base field $\Bbbk$ for some $r\in\Z_{>0}$ with $r L$ Cartier. 

If $F$ is dreamy over $(X, \Delta)$, then there exists the birational morphism 
$\sigma\colon Y\to X$ with $Y$ normal such that $F$ is a $\Q$-Cartier divisor on $Y$ 
and is anti-ample over $X$. We call the $\sigma$ the \emph{extraction} of $F$. 
\item\label{vst_dfn4}
The divisor $F$ is said to be \emph{plt-type} over $(X, \Delta)$ if there exists the 
extraction $\sigma\colon Y\to X$ of $F$ and the pair $(Y, \Delta_Y+F)$ is a plt pair, 
where the $\Q$-Weil divisor $\Delta_Y$ on $Y$ is defined by the equation 
\[
K_Y+\Delta_Y+\left(1-A_{X, \Delta}(F)\right)F=\sigma^*(K_X+\Delta).
\]
\item\label{vst_dfn5}
The divisor $F$ is said to be \emph{product-type} over $(X, \Delta)$ if there exists a 
one-parameter subgroup $\rho\colon\G_m\to\Aut(X, \Delta)$ such that the divisorial 
valuation $\ord_F\colon\Bbbk(X)^*\to\Z$ of $F$ is equal to the composition 
\[
\Bbbk(X)^*\xrightarrow{\rho^*}\Bbbk(X)(t)^*\xrightarrow{\ord_{(t^{-1})}}\Z,
\]
where the field extension 
$\rho^*$ is induced by the morphism $\rho\colon\G_m\times X\to X$. 
\end{enumerate}
\end{definition}

\begin{definition}\label{K_dfn}
\begin{enumerate}
\renewcommand{\theenumi}{\arabic{enumi}}
\renewcommand{\labelenumi}{(\theenumi)}
\item\label{K_dfn1}
\cite{li, vst}
$(X, \Delta)$ is said to be \emph{K-semistable} if $A_{X, \Delta}(F)\geq S(F)$ holds 
for any prime divisor $F$ over $X$. 
\item\label{K_dfn2}
\cite{vst, pltK}
$(X, \Delta)$ is said to be \emph{uniformly K-stable} if there exists $\delta>1$
such that 
$A_{X, \Delta}(F)\geq \delta\cdot S(F)$ holds
for any prime divisor $F$ over $X$. 
\item\label{K_dfn3}
\cite{hyperplane, BX}
$(X, \Delta)$ is said to be \emph{K-polystable} if K-semistable, and 
$A_{X, \Delta}(F)=S(F)$ holds for a prime divisor $F$ over $X$ only if $F$ is 
product-type over $(X, \Delta)$. 
\end{enumerate}
\end{definition}

\begin{remark}\label{dreamy_rmk}
It is known in \cite{vst, pltK, hyperplane, BX} 
that it is enough to check the above inequalities 
for only dreamy prime divisors over $(X, \Delta)$ in order to test 
the uniform K-stability, the K-polystability and the K-semistability. 
\end{remark}

As an immediate interpretation, we can get the following: 

\begin{definition}\label{delta_dfn}
We set 
\[
\delta(X, \Delta):=\inf_F\frac{A_{X, \Delta}(F)}{S(F)}, 
\]
and call it the \emph{delta-invariant} of $(X, \Delta)$, 
where $F$ runs through all prime divisors over $X$. 
\end{definition}

\begin{corollary}\label{delta_cor}
$(X, \Delta)$ is K-semistable $($resp., uniformly K-stable$)$ if and only if 
$\delta(X, \Delta)\geq 1$ $($resp., $\delta(X, \Delta)>1$$)$. 
\end{corollary}

\begin{remark}\label{basis_rmk}
The original definition of the delta-invariant in \cite{FO} is different from the above 
definition. 
In \cite{FO}, the definition of the delta-invariant relies on the notion of 
\emph{basis type divisors}. 
In \cite{BJ}, the authors showed that the above definition and 
the original definition are equivalent. In \cite{CRZ}, the authors evaluated the 
delta-invariant for some asymptotically log del Pezzo surfaces by analyzing 
basis type divisors. 
\end{remark}

We recall the following result obtained mainly in \cite{KpsMaeda}: 

\begin{proposition}\label{pltK_prop}
Let $(X, \Delta)$ be a $2$-dimensional log Fano pair. 
\begin{enumerate}
\renewcommand{\theenumi}{\arabic{enumi}}
\renewcommand{\labelenumi}{(\theenumi)}
\item\label{pltK_prop1}
If a prime divisor $F$ over $X$ satisfies that 
\[
\frac{A_{X, \Delta}(F)}{S(F)}<\frac{3}{2}, 
\]
then $F$ is dreamy over $(X, \Delta)$. 
In fact, the minimal resolution of the extraction of $F$ is a Mori dream space 
in the sense of \cite{HK}. 
\item\label{pltK_prop2}
If a prime divisor $F$ over $X$ is not plt-type over $(X, \Delta)$, then there exists 
a plt-type prime divisor $F'$ over $(X, \Delta)$ such that 
\[
\frac{A_{X, \Delta}(F)}{S(F)}>\frac{A_{X, \Delta}(F')}{S(F')} 
\]
and $c_X(F)\supset c_X(F')$, where $c_X(F)$ is the center of $F$ at $X$. 
\item\label{pltK_prop3}
Assume that an exceptional prime divisor $F$ over $X$ admits the extraction 
$\sigma\colon Y\to X$. Set 
\[
f:=-\frac{1}{(F^{\cdot 2})_Y}, \quad V:=((-(K_X+\Delta))^{\cdot 2})
\]
and 
\[
\vol(x):=\vol_Y\left(\sigma^*(-(K_X+\Delta))-xF\right). 
\]
\begin{enumerate}
\renewcommand{\theenumii}{\roman{enumii}}
\renewcommand{\labelenumii}{\rm{(\theenumii)}}
\item\label{pltK_prop31}
Assume that $\vol(\varepsilon)>0$ for some $\varepsilon\in\R_{>0}$. Then we have 
\[
\vol(x)\leq \frac{\left({\vol}'(\varepsilon)\right)^2}{4\vol(\varepsilon)}
\left(x-\varepsilon+\frac{2\vol(\varepsilon)}{{\vol}'(\varepsilon)}\right)^2
\]
for any $x\geq \varepsilon$. 
\item\label{pltK_prop32}
If $\sigma^*(-(K_X+\Delta))-\varepsilon F$ is nef on $Y$ for some 
$\varepsilon\in\R_{>0}$, then we have 
\[
S(F)\leq\frac{1}{3}\left(\frac{V f}{\varepsilon}+\varepsilon\right).
\]
\end{enumerate}
\end{enumerate}
\end{proposition}

\begin{proof}
\eqref{pltK_prop1} follows from \cite[Proposition 3.6]{KpsMaeda}. 

\eqref{pltK_prop2} follows from \cite[Corollary 3.2]{pltK}. 

\eqref{pltK_prop31} follows from \cite[Proposition 2.7]{KpsMaeda}. 

\eqref{pltK_prop32} follows from the argument in \cite[\S 4.2]{pltK}. We give the proof 
for the readers' convenience. For $x\in[0, \varepsilon]$, we have 
$\vol(x)=V-x^2/f$. If $\vol(\varepsilon)=0$, then the assertion is trivial. If 
$\vol(\varepsilon)>0$, then we have 
\begin{eqnarray*}
S(F)\leq\frac{1}{V}\left(\int_0^\varepsilon\left(V-\frac{x^2}{f}\right)dx+
\int_\varepsilon^{\frac{Vf}{\varepsilon}}\frac{\varepsilon^2}{f(Vf-\varepsilon^2)}
\left(x-\frac{Vf}{\varepsilon}\right)^2 dx\right)
=\frac{1}{3}\left(\frac{V f}{\varepsilon}+\varepsilon\right)
\end{eqnarray*}
by \eqref{pltK_prop31}. 
\end{proof}

\subsection{Plt-type prime divisors on surfaces}\label{plt_section}

In this section, we assume that $X$ is a smooth projective surface, $C$ is a nonzero 
smooth divisor on $X$, $\beta\in(0,1]\cap\Q$, and $F$ is an exceptional prime divisor 
over $X$ such that $F$ is plt-type over $(X, (1-\beta)C)$, that is, $F$ admits the 
extraction $\sigma\colon Y\to X$ and the pair $(Y, (1-\beta)C^Y+F)$ is a plt pair. 
We recall the notions in \cite[\S 3]{KpsMaeda}. 

\begin{definition}\label{monoidal_dfn}
\begin{enumerate}
\renewcommand{\theenumi}{\arabic{enumi}}
\renewcommand{\labelenumi}{(\theenumi)}
\item\label{monoidal_dfn1}
For the above $F$, we construct the sequence 
\[
\pi\colon \tilde{X}=X_m\to X_{m-1}\to\cdots\to X_1\to X_0=X
\]
of monoidal transforms inductively as follows: 
\begin{enumerate}
\renewcommand{\theenumii}{\roman{enumii}}
\renewcommand{\labelenumii}{\rm{(\theenumii)}}
\item\label{monoidal_dfn11}
If $F$ is exceptional over $X_i$, then let $\pi_{i+1}\colon X_{i+1}\to X_i$ be the 
blowup along $p_{i+1}:=c_{X_i}(F)\in X_i$ and let $F_{i+1}\subset X_{i+1}$ be the 
exceptional divisor of $\pi_{i+1}$. 
\item\label{monoidal_dfn12}
If $F\subset X_i$, then we set $m:=i$ and we stop the construction. 
\end{enumerate}
We call the above the \emph{sequence of monoidal transforms with respects to $F$}. 
The morphism factors through the morphism $\sigma$. The induced morphism 
$\nu\colon\tilde{X}\to Y$ is nothing but the minimal resolution of $Y$ (see 
\cite[Lemma 3.3 (4)]{KpsMaeda}). 
\item\label{monoidal_dfn2}
For any $2\leq i\leq m$, we set the integer $0\leq q(i)\leq i-2$ as follows: 
\begin{enumerate}
\renewcommand{\theenumii}{\roman{enumii}}
\renewcommand{\labelenumii}{\rm{(\theenumii)}}
\item\label{monoidal_dfn21}
If $p_i\in F_j^{X_{i-1}}$ for some $1\leq j\leq i-2$, we set $q(i):=j$. 
\item\label{monoidal_dfn22}
Otherwise, we set $q(i):=0$. 
\end{enumerate}
Moreover, we set 
\[
k:=\max\{2\leq i\leq m\,|\,q(i)=0\}. 
\]
Since $F$ is plt-type over $(X, (1-\beta)C)$, we have $q(i)=0$ for any $2\leq i\leq k$ 
(see \cite[Definition 3.4 (1)]{KpsMaeda}). 
\item\label{monoidal_dfn3}
For any $0\leq i\leq m$, we inductively define 
\begin{enumerate}
\renewcommand{\theenumii}{\roman{enumii}}
\renewcommand{\labelenumii}{\rm{(\theenumii)}}
\item\label{monoidal_dfn31}
$(a_0$, $b_0):=(1$, $0)$, \,\, $(a_1$, $b_1):=(1$, $1)$, 
\item\label{monoidal_dfn32}
$(a_i$, $b_i):=(a_{q(i)}$, $b_{q(i)})+(a_{i-1}$, $b_{i-1})$ for $2\leq i\leq m$. 
\end{enumerate}
We set $(a$, $b):=(a_m$, $b_m)$. 
\item\label{monoidal_dfn4}
For any $0\leq i\leq m$, we define the effective divisor $F_i^*$ on $X_i$ inductively 
as follows: 
\begin{enumerate}
\renewcommand{\theenumii}{\roman{enumii}}
\renewcommand{\labelenumii}{\rm{(\theenumii)}}
\item\label{monoidal_dfn41}
$F_0^*:=0$ (on $X=X_0$), \,\, $F_1^*:=F_1$ on $X_1$, 
\item\label{monoidal_dfn41}
\[
F_i^*:=\left(\pi_{q(i)+1}\circ\cdots\circ\pi_i\right)^*F_{q(i)}^*+\pi_i^*F_{i-1}^*+F_i
\]
on $X_i$ for $2\leq i\leq m$. 
\end{enumerate}
\end{enumerate}
\end{definition}

We recall the following result obtained in \cite{KpsMaeda}.

\begin{proposition}\label{KpM_prop}
\begin{enumerate}
\renewcommand{\theenumi}{\arabic{enumi}}
\renewcommand{\labelenumi}{(\theenumi)}
\item\label{KpM_prop1}
\cite[Lemmas 3.5 (3) and 3.3 (4)]{KpsMaeda}
We have 
\[
\left(F^{\cdot 2}\right)_Y=-\frac{1}{ab} \quad\text{ and }\quad F_m^*=ab\cdot \nu^*F.
\]
\item\label{KpM_prop2}
\cite[Lemma 3.5 (2) and Theorem 5.1 Step 7]{KpsMaeda}
If $p_1\not\in C$, then we set $j_C:=0$. Otherwise, we set 
\[
j_C:=\max\left\{1\leq i\leq k\,|\,p_i\in C^{X_{i-1}}\right\}.
\]
Then we have 
\[
A_{X, \Delta}(F)=a+b-\min\{j_C b,\,\,a\}(1-\beta). 
\]
\item\label{KpM_prop3}
\cite[Lemma 3.5 (1)]{KpsMaeda}
For any $1\leq i\leq m$, $a_i$ and $b_i$ are mutually prime. 
\item\label{KpM_prop4}
\cite[Lemma 3.5 (3)]{KpsMaeda}
We have $k=\lceil a/b\rceil$, that is, $k$ is the minimum integer which is not less than 
$a/b$. 
\item\label{KpM_prop5}
\cite[Lemma 3.5 (4)]{KpsMaeda}
For any $1\leq i\leq k$, we have 
\[
\coeff_{F_i^{\tilde{X}}}F_m^*=\min\{i b, \, a\}.
\]
\end{enumerate}
\end{proposition}

We use the following lemma many times in \S \ref{S7_section}--\S \ref{caseII_section}. 

\begin{lemma}\label{mono_lem}
For any smooth curve $B\subset X$ with $p_1\in B$, let us set 
\[
j_B:=\max\left\{1\leq i\leq k\,|\,p_i\in B^{X_{i-1}}\right\}
\]
as in Proposition \ref{KpM_prop} \eqref{KpM_prop2}. 
\begin{enumerate}
\renewcommand{\theenumi}{\arabic{enumi}}
\renewcommand{\labelenumi}{(\theenumi)}
\item\label{mono_lem1}
If $p_1\in C$ and $\beta\leq 1/2$, then we have $j_C=1$ or $k$. 
\item\label{mono_lem2}
We have 
\[
\left(B^Y\cdot F\right)_Y=\begin{cases}
\frac{j_B}{a} & \text{if }j_B<k, \\
\frac{1}{b} & \text{if }j_B=k. 
\end{cases}\]
\item\label{mono_lem3}
We have 
\[
\left(B^{\cdot 2}\right)_X-\left(B^{Y\cdot 2}\right)_Y=\begin{cases}
\frac{j_B^2 b}{a} & \text{if }j_B<k, \\
\frac{a}{b} & \text{if }j_B=k. 
\end{cases}\]
\item\label{mono_lem4}
\begin{enumerate}
\renewcommand{\theenumii}{\roman{enumii}}
\renewcommand{\labelenumii}{\rm{(\theenumii)}}
\item\label{mono_lem41}
If $j_B<k$, then we have
\[
\coeff_{F_i^{\tilde{X}}}\nu^*\left(B^Y\right)=\frac{i}{a}(a-j_B b)
\]
for any $1\leq i\leq j_B$. 
\item\label{mono_lem42}
If $j_B=k$, then we have 
\[
\coeff_{F_i^{\tilde{X}}}\nu^*\left(B^Y\right)=\begin{cases}
0 & \text{if }1\leq i\leq k-1,\\
k-\frac{a}{b} & \text{if }i=k.
\end{cases}\]
\end{enumerate}
\end{enumerate}
\end{lemma}

\begin{proof}
Assume that $j_B=k$. The assertions \eqref{mono_lem2}, \eqref{mono_lem3} and 
\eqref{mono_lem4} are \'etale local around a neighborhood of $p_1\in(X, B)$. 
Thus, as in the proof of \cite[Lemma 3.5]{KpsMaeda}, we may assume that: 
\begin{itemize}
\item
$X$ corresponds to the complete fan $\Sigma$ in $N_\R$ with the set of 
$1$-dimensional cones equal to the set 
\[
\{\R_{\geq 0}(0,1), \,\,\R_{\geq 0}(-1,-1), \,\,\R_{\geq 0}(1,0)\},
\]
$B$ corresponds the $1$-dimensional cone $\R_{\geq 0}(1,0)$ in $\Sigma$,
\item
$Y$ corresponds to the complete fan $\Sigma'$ in $N_\R$ with the set of 
$1$-dimensional cones equal to the set 
\[
\{\R_{\geq 0}(0,1), \,\,\R_{\geq 0}(-1,-1), \,\,\R_{\geq 0}(1,0), \,\,\R_{\geq 0}(a,b)\},
\]
$F$ corresponds the $1$-dimensional cone $\R_{\geq 0}(a,b)$ in $\Sigma'$.
\end{itemize}
Then the assertions \eqref{mono_lem2} and \eqref{mono_lem3} are well-known. 
Let us consider the assertion \eqref{mono_lem4}. We can easily check that 
$\coeff_{F_i^{\tilde{X}}}\nu^*\left(B^Y\right)=0$ for any $1\leq i\leq k-1$ by looking at 
the dual graph of the union of $\pi$-exceptional curves and $B^{\tilde{X}}$. From 
\eqref{mono_lem2}, we have 
\begin{eqnarray*}
\left(B^{{\tilde{X}}\cdot 2}\right)+k-\frac{a}{b}=(B^{\cdot 2})-\frac{a}{b}
=\left(B^{Y\cdot 2}\right)
=\left(\nu^*\left(B^Y\right)\cdot B^{\tilde{X}}\right)
=\left(B^{\tilde{X}\cdot 2}\right)+\coeff_{F_k^{\tilde{X}}}\nu^*\left(B^Y\right).
\end{eqnarray*}

Assume that $j_B<k$. For any $j_B\leq i\leq m$, we define the effective divisor 
$B_i^*$ on $X_i$ as follows: 
\begin{itemize}
\item
For any $j_B\leq i\leq k$, we set
\[
B_i^*:=\sum_{l=1}^{j_B}(i-j_B)l F_l^{X_i}+\sum_{l=j_B+1}^{i-1}(i-l)j_B F_l^{X_i}+i B^{X_i},
\]
where we add nothing for the summation $\sum_{l=j_B+1}^{i-1}$ when 
$i=j_B$ or $i=j_B+1$. 
\item
For any $k+1\leq i\leq m$, we set 
\[
B_i^*:=\left(\pi_{q(i)+1}\circ\cdots\circ\pi_i\right)^*B_{q(i)}^*+\pi_i^*B_{i-1}^*-j_B F_i. 
\]
(Since $q(i)\geq k-1$ for any $i\geq k+1$, the definition makes sense.) 
\end{itemize}
We can inductively check that 
\begin{eqnarray*}
\left(B_i^*\cdot F_l^{X_i}\right)_{X_i} & = & \begin{cases}
0 & \text{if } l<i, \\
j_B & \text{if } l=i, 
\end{cases}\\
\coeff_{B^{X_i}}B_i^* &=& a_i, \\
\coeff_{F_l^{X_i}}B_i^*&=& l(a_i-j_B b_i)\,\, \text{ for any }\,\, 1\leq l\leq j_B.
\end{eqnarray*}
In particular, we have $\nu^*\left(B^Y\right)=(1/a)B_m^*$. Thus we get 
\begin{eqnarray*}
\left(B^Y\cdot F\right)_Y &=& 
\frac{1}{a}\left(B_m^*\cdot F_m^{X_m}\right)_{X_m}=\frac{j_B}{a}, \\
\left(B^{Y\cdot 2}\right)_Y& = &\frac{1}{a}\left(B_m^*\cdot B^{\tilde{X}}\right)
=\left(B^{\tilde{X}\cdot 2}\right)+\frac{1}{a}\coeff_{F_{j_B}^{\tilde{X}}}B_m^* \\
&=&\left(B^{\cdot 2}\right)_X-\frac{j_B^2 b}{a}.
\end{eqnarray*}

The remaining assertion is only \eqref{mono_lem1}. Assume that $j_C<k$. Since 
$\coeff_{F_{j_C}^{\tilde{X}}}K_{\tilde{X}/X}=j_C$ (see \cite[Lemma 3.5 (2)]{KpsMaeda}) and 
\begin{eqnarray*}
&&\nu^*\left(K_Y+(1-\beta)C^Y+F\right)\\
&=&\nu^*\left(\sigma^*K_X+(1-\beta)C^Y+(a+b)F\right)\\
&=&K_{\tilde{X}}-K_{\tilde{X}/X}+\frac{a+b}{ab}F_m^*+\frac{1-\beta}{a}C_m^*,
\end{eqnarray*}
we get 
\[
A_{Y, (1-\beta)C^Y+F}\left(F_{j_C}\right)
=1+j_C-\frac{j_C}{a}(a+b)-\frac{1-\beta}{a}j_C(a-j_C b)
\]
by Proposition \ref{KpM_prop}. 
Since $F$ is plt-type over $(X, (1-\beta)C)$, we have 
\[
0<A_{Y, (1-\beta)C^Y+F}\left(F_{j_C}\right)=1-\frac{j_C}{a}\left(b+(1-\beta)(a-j_C b)\right).
\]
This implies that $j_C<1/(1-\beta)$. Since $\beta\leq 1/2$, we get $j_C=1$. 
\end{proof}

\subsection{Basic properties of asymptotically log del Pezzo surfaces}\label{dP_section}

In this section, we assume that $(X, C)$ is an asymptotically log del Pezzo surface 
of type $(\operatorname{I.9B.}n)$, $\eta\colon X\to \pr^1$ is the 
anti-log-canonical morphism, and $q_1$, $q_2\in C$ is the ramification points of 
$\eta|_C\colon C\to \pr^1$. We remark the following easy lemma. 

\begin{lemma}\label{easy_lem}
\begin{enumerate}
\renewcommand{\theenumi}{\arabic{enumi}}
\renewcommand{\labelenumi}{(\theenumi)}
\item\label{easy_lem1}
For any $\beta\in(0$, $1)$, we have $-(K_X+(1-\beta)C)\sim_\R\beta C+l$, where 
$l$ is a fiber of $\eta$. 
\item\label{easy_lem2}
For any $\beta\in(0$, $1)\cap\Q$, we have 
\[
\left(-(K_X+(1-\beta)C)\cdot C\right)=(4-n)\beta+2.
\]
In particular, the pair $(X, (1-\beta)C)$ is a $2$-dimensional log Fano pair if and only if 
$(4-n)\beta+2>0$. 
\item\label{easy_lem3}
For any birational morphism $\theta\colon X\to X'$ over $\pr^1$ with the Picard rank 
of $X'$ bigger than $2$ obtained by contracting numbers of $(-1)$-curves on 
$X\xrightarrow{\eta}\pr^1$. Then the pair $(X', \theta_*C)$ is also an 
asymptotically log del Pezzo surface of type $(\operatorname{I.9B.}n')$, where 
$n'+2$ is the Picard rank of $X'$. 
\end{enumerate}
\end{lemma}

\begin{proof}
\eqref{easy_lem1} is trivial. \eqref{easy_lem2} follows from Nakai's criterion for 
ampleness. For \eqref{easy_lem3}, for any $(-1)$-curve $l_1\subset X$ with 
$\eta_*l_1=0$, we have $(C\cdot l_1)=1$ and there exists another $(-1)$-curve 
$l_2\subset X$ with $\eta_*l_2=0$ such that $l_1+l_2$ is a fiber of $\eta$. 
Thus the assertion follows. 
\end{proof}

The following proposition is easy but important in this article. 

\begin{proposition}\label{easy_prop}
Take any $\beta\in(0$, $1)\cap\Q$ with $(4-n)\beta+2>0$ and set 
$L:=-(K_X+(1-\beta)C)$. Consider a prime divisor $F$ over $X$. Take any birational 
morphism $\theta\colon X\to X'$ over $\pr^1$ with the Picard rank of $X'$ bigger 
than $2$ obtained by contracting numbers of $(-1)$-curves on 
$X\xrightarrow{\eta}\pr^1$ as in Lemma \ref{easy_lem} \eqref{easy_lem3}. If 
$\theta$ is an isomorphism at the generic point of $c_X(F)$, then we have 
\[
S_L(F)\leq \frac{\left({L'}^{\cdot 2}\right)}{\left(L^{\cdot 2}\right)}S_{L'}(F),
\]
where $L':=-\left(K_{X'}+(1-\beta)\theta_*C\right)$. 
\end{proposition}

\begin{proof}

By Lemma \ref{easy_lem}, the pair $(X', \theta_*C)$ is an asymptotically log del Pezzo 
surface and $L'$ is ample. Since $(X', (1-\beta)\theta_*C)$ has only terminal 
singularities, we get 
\[
K_X+(1-\beta)C\geq \theta^*(K_{X'}+(1-\beta)\theta_*C). 
\]
Thus $\theta^*L'-L$ is effective. Hence the assertion follows. 
Indeed, we have  
\[
\vol_{\tilde{X}}\left(\pi^*\theta^*L'-x F\right)\geq \vol_{\tilde{X}}\left(\pi^*L-x F\right)
\]
for any $x\in\R_{\geq 0}$. 
\end{proof}

\section{On the del Pezzo surface of degree seven}\label{S7_section}

In this section, we prove the following: 

\begin{thm}\label{S7_thm}
Let $(X, C)$ be an asymptotically log del Pezzo surface of type $(\operatorname{I.9B.}1)$, 
let $\eta\colon X\to \pr^1$ be the anti-log-canonical morphism, and let 
$q_1$, $q_2\in C$ be the ramification points of $\eta|_C\colon C\to\pr^1$. Take any 
prime divisor $F$ over $X$ such that $c_X(F)\not\in\{q_1$, $q_2\}$. Then, for any 
$\beta\in(0$, $1/7)\cap\Q$, we have 
\[
\frac{A_{X, \Delta}(F)}{S_L(F)}\geq \frac{6}{5}, 
\]
where $\Delta:=(1-\beta)C$ and $L:=-(K_X+\Delta)$. 
\end{thm}

\begin{proof}
The following proof is divided into 13 numbers of steps. 

\noindent\underline{\textbf{Step 1}}\\
Since $X$ is the del Pezzo surface of degree $7$, $X$ corresponds to the complete fan 
in $N_\R$ whose set of $1$-dimensional cones is equal to the set 
\[
\{\R_{\geq 0}(1, 0),\,\,\R_{\geq 0}(1, 1),\,\,\R_{\geq 0}(0, 1),\,\,\R_{\geq 0}(-1, 0),\,\,
\R_{\geq 0}(0, -1)\}.
\]
Let $E_1$, $E$, $E_2$, $E_1^\infty$, $E_2^\infty$ be the torus invariant prime divisor 
on $X$ corresponds to 
$\R_{\geq 0}(1, 0)$, $\R_{\geq 0}(1, 1)$, $\R_{\geq 0}(0, 1)$, 
$\R_{\geq 0}(-1, 0)$, $\R_{\geq 0}(0, -1)$, respectively. 
We can assume that $C\sim E_1+2E_2+2E$. 
Then $L$ is $\Q$-linearly equivalent to the torus invariant $\Q$-divisor 
\[
D:=(\beta+1)E_1+2\beta E_2+(2\beta+1)E.
\]
The $\Q$-divisor $D$ corresponds to the polytope $P\subset M_\R$ defined by the 
set of $m\in M_\R$ with 
\begin{eqnarray*}
\langle m, (1,0)\rangle &\geq & -\beta-1, \\
\langle m, (1,1)\rangle &\geq & -2\beta-1, \\
\langle m, (0,1)\rangle &\geq & -2\beta, \\
\langle m, (-1,0)\rangle &\geq & 0, \\
\langle m, (0,-1)\rangle &\geq & 0.
\end{eqnarray*}
As in \cite[Remark 2.9]{FLSZZ}, the barycenter of $P$ is equal to 
\[
\left(-\frac{4\beta^2+9\beta+6}{3(3\beta+4)}, \,\,
-\frac{7\beta^2+12\beta}{3(3\beta+4)}\right). 
\]
Therefore, by \cite[Corollary 7.16]{BJ}, we have 
\[
\inf_{F'}\frac{A_{X, 0}(F')}{S_L(F')}=\frac{A_{X, 0}(E)}{S_L(E)}
=\frac{3(3\beta+4)}{7\beta^2+12\beta+6},
\]
where $F'$ runs through all prime divisors over $X$. 

\noindent\underline{\textbf{Step 2}}\\
Assume that $F=C$. Then $A_{X, \Delta}(C)=\beta$. Since 
$L-xC\sim_\R(\beta-x)C+E_1+E$ is nef for any $x\in[0$, $\beta]$, we have
\begin{eqnarray*}
S(C)&=&\frac{1}{3\beta^2+4\beta}\int_0^\beta\left((L-x C)^{\cdot 2}\right)dx\\
&=&\frac{1}{3\beta^2+4\beta}\int_0^\beta\left(3(\beta-x)^2+4(\beta-x)\right)dx
=\frac{\beta(\beta+2)}{3\beta+4}.
\end{eqnarray*}
Thus we have 
\[
\frac{A_{X,\Delta}(C)}{S(C)}=\frac{3\beta+4}{\beta+2}>\frac{6}{5}.
\]

\noindent\underline{\textbf{Step 3}}\\
Assume that $c_X(F)\not\subset C$. Then $A_{X, \Delta}(F)=A_{X, 0}(F)$. Thus, by Step 1, 
we have
\[
\frac{A_{X, \Delta}(F)}{S(F)}\geq\frac{3(3\beta+4)}{7\beta^2+12\beta+6}>\frac{6}{5}. 
\]

\noindent\underline{\textbf{Step 4}}\\
By Steps 2, 3 and Proposition \ref{pltK_prop}, we may assume that $F$ is an exceptional, 
dreamy and plt-type prime divisor over $(X, \Delta)$ with $c_X(F)\in C$. 
We follow the notations in \S \ref{plt_section}. Moreover, let us set 
\[
j_C:=\max\left\{1\leq i\leq k\,|\,p_i\in C^{X_{i-1}}\right\}. 
\]
By Lemma \ref{mono_lem}, we have $j_C\in\{1$, $k\}$. From Proposition \ref{KpM_prop}, 
we have 
\[
A_{X, \Delta}(F)=\begin{cases}
b\beta+a & \text{if }j_C=1, \\
a\beta+b & \text{if }j_C=k.
\end{cases}\]

Assume that $j_C=1$ and $k\geq 3$. Then we have 
\[
\frac{A_{X, \Delta}(F)}{A_{X, 0}(F)}=\frac{\beta+(a/b)}{1+(a/b)}>\frac{\beta+2}{3}
\]
since $a/b>2$. By Step 1, we have 
\[
\frac{A_{X, \Delta}(F)}{S(F)}>\frac{(\beta+2)(3\beta+4)}{7\beta^2+12\beta+6}>\frac{6}{5}.
\]
Thus we may assume that either $j_C=k$ or $(j_C$, $k)=(1$, $2)$. 

\noindent\underline{\textbf{Step 5}}\\
Assume that $p_1\in C\setminus(E\cup E_1)$. In this case, we may assume that 
$p_1=E_1^\infty\cap E_2^\infty$. The curves $E_1^\infty$ and $C$ intersect 
transversally at $p_1$ since $p_1\not\in\{q_1$, $q_2\}$. 

We consider the case that $F$ is toric. In this case, $F$ corresponds to a primitive lattice 
point $v_F\in N$. Since $C$ intersects $E_1^\infty$ and $E_2^\infty$ transversally 
at $p_1$, we must have $j_C=1$. One of the following holds: 
\begin{enumerate}
\renewcommand{\theenumi}{\arabic{enumi}}
\renewcommand{\labelenumi}{(\theenumi)}
\item\label{step5-1}
$v_F=(-a$, $-b)$, 
\item\label{step5-2}
$v_F=(-b$, $-a)$.
\end{enumerate}
For the case \eqref{step5-1}, by \cite[Corollary 7.7]{BJ}, we have 
\[
S(F)=\frac{a(4\beta^2+9\beta+6)+b(7\beta^2+12\beta)}{3(3\beta+4)}. 
\]
Thus we have 
\[
\frac{A_{X, \Delta}(F)}{S(F)}=\frac{3(3\beta+4)(\beta+(a/b))}
{(a/b)(4\beta^2+9\beta+6)+7\beta^2+12\beta}
\geq\frac{3(3\beta+4)(\beta+1)}{11\beta^2+21\beta+6}>\frac{6}{5}
\]
since $a/b\in[1$, $2]$. 
For the case \eqref{step5-2}, by \cite[Corollary 7.7]{BJ}, we have 
\[
S(F)=\frac{a(7\beta^2+12\beta)+b(4\beta^2+9\beta+6)}{3(3\beta+4)}. 
\]
Thus we have 
\[
\frac{A_{X, \Delta}(F)}{S(F)}=\frac{3(3\beta+4)(\beta+(a/b))}
{(a/b)(7\beta^2+12\beta)+4\beta^2+9\beta+6}
\geq\frac{3(3\beta+4)(\beta+1)}{11\beta^2+21\beta+6}>\frac{6}{5}
\]
since $a/b\in[1$, $2]$. 

\noindent\underline{\textbf{Step 6}}\\
Assume that $p_1=E_1^\infty\cap E_2^\infty$ and $j_C=1$. By Step 5, we may assume that 
$k=2$ and $p_2\in F_1\setminus\left(C^{X_1}\cup(E_1^\infty)^{X_1}
\cup(E_2^\infty)^{X_1}\right)$. Set 
\[
D_1:=(\beta+1)E_1^\infty+\beta E_2^\infty+\beta E_2\sim_\Q L.
\]
Then we can inductively check that 
\[
\pi^*D_1=D_1^{\tilde{X}}+\sum_{i=1}^m b_i(2\beta+1)F_i^{\tilde{X}}. 
\]
By Proposition \ref{KpM_prop}, 
$\pi^*D_1-(x/(ab))F_m^*=\nu^*(\sigma^*D_1-x F)$ is effective for 
$x\in[0$, $b(2\beta+1)]$. We have 
\begin{eqnarray*}
\left(\pi^*L-\frac{x}{ab}F_m^*\right)\cdot\left(E_1^\infty\right)^{\tilde{X}}&=&
2\beta-\frac{x}{a}, \\
\left(\pi^*L-\frac{x}{ab}F_m^*\right)\cdot\left(E_2^\infty\right)^{\tilde{X}}&=&
\beta+1-\frac{x}{a}, \\
\left(\pi^*L-\frac{x}{ab}F_m^*\right)\cdot E_2^{\tilde{X}}&=&
1, \\
\left(\pi^*L-\frac{x}{ab}F_m^*\right)\cdot F^{\tilde{X}}&=&
\frac{x}{ab}.
\end{eqnarray*}
Note that $a(\beta+1)>2a\beta$ and $b(2\beta+1)>2a\beta$. Thus 
$\sigma^*L-x F$ is nef for $x\in[0$, $2a\beta]$, and $\sigma^*L-2a\beta F$ induces 
the birational contraction $\phi\colon Y\to Y'$ of $(E_1^\infty)^Y$. Set 
$\nu':=\phi\circ\nu\colon\tilde{X}\to Y'$. 
By Lemma \ref{mono_lem}, we have $\left((E_1^\infty)^{Y\cdot 2}\right)=-b/a$. 
Thus we have 
\begin{equation}\label{6-1_eqn}
\pi^*D_1-\frac{x}{ab}F_m^*={\nu'}^*\nu'_*\left(\pi^*D_1-\frac{x}{ab}F_m^*\right)
+\frac{x-2a\beta}{b}\nu^*\left((E_1^\infty)^Y\right).
\end{equation}
For any $x\in[2a\beta$, $b(2\beta+1)]$, we have 
$\nu'_*\left(\pi^*D_1-(x/(ab))F_m^*\right)\geq 0$ and 
\begin{eqnarray*}
\left(\pi^*L-\frac{x}{ab}F_m^*-\frac{x-2a\beta}{b}\nu^*\left((E_1^\infty)^Y\right)
\right)\cdot\left(E_2^\infty\right)^{\tilde{X}}&=&
\frac{(2a-b)\beta+b-x}{b}, \\
\left(\pi^*L-\frac{x}{ab}F_m^*-\frac{x-2a\beta}{b}\nu^*\left((E_1^\infty)^Y\right)
\right)\cdot E_2^{\tilde{X}}&=&
\frac{2a\beta+b-x}{b}, \\
\left(\pi^*L-\frac{x}{ab}F_m^*-\frac{x-2a\beta}{b}\nu^*\left((E_1^\infty)^Y\right)
\right)\cdot F^{\tilde{X}}&=&
\frac{2\beta}{b}.
\end{eqnarray*}
Note that $b\beta+b<\min\{b(2\beta+1)$, $(2a-b)\beta+b$, $2a\beta+b\}$. Thus, 
for any $x\in[2a\beta$, $b\beta+b]$, \eqref{6-1_eqn} gives the Zariski decomposition, 
hence
\begin{eqnarray*}
\vol_Y(\sigma^*L-x F)&=&\left(\pi^*L-\frac{x}{ab}F_m^*-
\frac{x-2a\beta}{b}\nu^*\left((E_1^\infty)^Y\right)\right)^{\cdot 2}\\
&=&3\beta^2+4\beta-\frac{x^2}{ab}+\frac{(x-2a\beta)^2}{ab}.
\end{eqnarray*}
By Proposition \ref{pltK_prop}, for $x\geq b\beta+b$, 
\[
\vol_Y(\sigma^*L-x F)\leq \frac{4}{b(4a-b)}\left(x-\left(b \beta+b
+\frac{(4a-b)\beta}{2}\right)\right)^2.
\]
In particular, we have 
\begin{eqnarray*}
S(F)&\leq&\frac{1}{3\beta^2+4\beta}\biggl(\int_0^{b\beta+b}\left(3\beta^2
+4\beta-\frac{x^2}{ab}\right)dx+\int_{2a\beta}^{b\beta+b}\frac{(x-2a\beta)^2}{ab}dx\\
&&+\int_{b\beta+b}^{b\beta+b+\frac{(4a-b)\beta}{2}}\frac{4}{b(4a-b)}
\left(x-\left(b\beta+b+\frac{(4a-b)\beta}{2}\right)\right)^2dx\biggr)\\
&=&\frac{a(16\beta^2+24\beta)+b(7\beta^2+18\beta+12)}{6(3\beta+4)}. 
\end{eqnarray*}
Thus we have 
\begin{eqnarray*}
\frac{A_{X, \Delta}(F)}{S(F)}&\geq&
\frac{6(3\beta+4)(\beta+(a/b))}{(a/b)(16\beta^2+24\beta)
+7\beta^2+18\beta+12}\\
&>&\frac{6(3\beta+4)(\beta+1)}{23\beta^2+42\beta+12}>\frac{6}{5}
\end{eqnarray*}
since $a/b\in(1$, $2]$. 

\noindent\underline{\textbf{Step 7}}\\
Assume that $p_1=E_1^\infty\cap E_2^\infty$ and $j_C=k\geq 2$. Set 
\[
D_2:=\beta C + E_1^\infty\sim_\Q L.
\]
Then we can inductively check that 
\[
\pi^*D_2=D_2^{\tilde{X}}+\sum_{i=1}^m (a_i\beta+b_i)F_i^{\tilde{X}}. 
\]
By Proposition \ref{KpM_prop}, 
$\pi^*D_2-(x/(ab))F_m^*=\nu^*(\sigma^*D_2-x F)$ is effective for 
$x\in[0$, $a\beta+b]$. We have 
\begin{eqnarray*}
\left(\pi^*L-\frac{x}{ab}F_m^*\right)\cdot C^{\tilde{X}}&=&
3\beta+2-\frac{x}{b}, \\
\left(\pi^*L-\frac{x}{ab}F_m^*\right)\cdot\left(E_1^\infty\right)^{\tilde{X}}&=&
2\beta-\frac{x}{a}, \\
\left(\pi^*L-\frac{x}{ab}F_m^*\right)\cdot F^{\tilde{X}}&=&
\frac{x}{ab}.
\end{eqnarray*}

\underline{The case $a/b\geq (3\beta+1)/\beta$}\,\,
We have $3b\beta+2b\leq\min\{a\beta+b$, $2a\beta\}$ in this case. 
Thus 
$\sigma^*L-x F$ is nef for $x\in[0$, $3b\beta+2b]$. By Proposition \ref{pltK_prop}, 
\[
S(F)\leq\frac{a(3\beta^2+4\beta)+b(3\beta+2)^2}{3(3\beta+2)}. 
\]
Thus we have 
\begin{eqnarray*}
\frac{A_{X, \Delta}(F)}{S(F)}&\geq&
\frac{3(3\beta+2)((a/b)\beta+1)}{(a/b)(3\beta^2+4\beta)
+(3\beta+2)^2}\\
&\geq&\frac{3(3\beta+2)^2}{18\beta^2+27\beta+8}>\frac{6}{5}
\end{eqnarray*}
since $a/b\in[(3\beta+1)/\beta$, $\infty)$. 

\underline{The case $1/\beta\leq a/b< (3\beta+1)/\beta$}\,\,
We have 
$a\beta+b\leq\min\{3b\beta+2b$, $2a\beta\}$ in this case. 
Thus 
$\sigma^*L-x F$ is nef for $x\in[0$, $a\beta+b]$. By Proposition \ref{pltK_prop}, 
\[
S(F)\leq\frac{ab(3\beta^2+4\beta)+(a\beta+b)^2}{3(a\beta+b)}. 
\]
Thus we have 
\begin{eqnarray*}
\frac{A_{X, \Delta}(F)}{S(F)}\geq
\frac{3((a/b)\beta+1)^2}{((a/b)\beta+1)^2+(a/b)(3\beta^2+4\beta)}
>\frac{12}{3\beta+8}>\frac{6}{5}
\end{eqnarray*}
since $a/b\in[1/\beta$, $(3\beta+1)/\beta)$. 

\underline{The case $a/b<1/\beta$}\,\,
We have $2a\beta<\min\{a\beta+b$, $3b\beta+2b\}$ in this case. 
Thus 
$\sigma^*L-x F$ is nef for $x\in[0$, $2a\beta]$, 
and $\sigma^*L-2a\beta F$ induces 
the birational contraction $\phi\colon Y\to Y'$ of $(E_1^\infty)^Y$. Set 
$\nu':=\phi\circ\nu\colon\tilde{X}\to Y'$. 
By Lemma \ref{mono_lem}, we have $\left((E_1^\infty)^{Y\cdot 2}\right)=-b/a$. 
Thus we have 
\begin{equation}\label{7-1_eqn}
\pi^*D_2-\frac{x}{ab}F_m^*={\nu'}^*\nu'_*\left(\pi^*D_2-\frac{x}{ab}F_m^*\right)
+\frac{x-2a\beta}{b}\nu^*\left((E_1^\infty)^Y\right).
\end{equation}
For any $x\in[2a\beta$, $a\beta+b]$, we have 
$\nu'_*\left(\pi^*D_2-(x/(ab))F_m^*\right)\geq 0$ and 
\begin{eqnarray*}
\left(\pi^*L-\frac{x}{ab}F_m^*-\frac{x-2a\beta}{b}\nu^*\left((E_1^\infty)^Y\right)
\right)\cdot C^{\tilde{X}}&=&
\frac{2(a\beta+(3/2)b\beta+b-x)}{b}, \\
\left(\pi^*L-\frac{x}{ab}F_m^*-\frac{x-2a\beta}{b}\nu^*\left((E_1^\infty)^Y\right)
\right)\cdot F^{\tilde{X}}&=&
\frac{2\beta}{b}.
\end{eqnarray*}
Thus, 
for any $x\in[2a\beta$, $a\beta+b]$, \eqref{7-1_eqn} gives the Zariski decomposition, 
hence
\begin{eqnarray*}
\vol_Y(\sigma^*L-x F)&=&\left(\pi^*L-\frac{x}{ab}F_m^*-
\frac{x-2a\beta}{b}\nu^*\left((E_1^\infty)^Y\right)\right)^{\cdot 2}\\
&=&3\beta^2+4\beta-\frac{x^2}{ab}+\frac{(x-2a\beta)^2}{ab}.
\end{eqnarray*}
By Proposition \ref{pltK_prop}, for $x\geq a\beta+b$, 
\[
\vol_Y(\sigma^*L-x F)\leq \frac{4}{3b^2}\left(x-\left(a \beta+b
+\frac{3b\beta}{2}\right)\right)^2.
\]
In particular, we have 
\begin{eqnarray*}
S(F)&\leq&\frac{1}{3\beta^2+4\beta}\biggl(\int_0^{a\beta+b}\left(3\beta^2
+4\beta-\frac{x^2}{ab}\right)dx+\int_{2a\beta}^{a\beta+b}\frac{(x-2a\beta)^2}{ab}dx\\
&&+\int_{a\beta+b}^{a\beta+b+\frac{3b\beta}{2}}\frac{4}{3b^2}
\left(x-\left(a\beta+b+\frac{3b\beta}{2}\right)\right)^2dx\biggr)\\
&=&\frac{-4a^2\beta^2+ab(18\beta^2+24\beta)+b^2(9\beta^2+18\beta+12)}{6b(3\beta+4)}. 
\end{eqnarray*}
Thus we have 
\begin{eqnarray*}
\frac{A_{X, \Delta}(F)}{S(F)}&\geq&
\frac{6(3\beta+4)((a/b)\beta+1)}{-4(a/b)^2\beta^2+(a/b)(18\beta^2+24\beta)
+9\beta^2+18\beta+12}\\
&>&\frac{6(3\beta+4)((a/b)\beta+1)}{-3(a/b)^2\beta^2+(a/b)(18\beta^2+24\beta)
+9\beta^2+18\beta+12}\\
&>&\frac{4(3\beta+4)}{3\beta^2+12\beta+11}>\frac{6}{5}
\end{eqnarray*}
since $a/b\in(1$, $1/\beta)$. 

\noindent\underline{\textbf{Step 8}}\\
Assume that $p_1\in C\cap(E\cup E_1)$. Since $p_1\not\in\{q_1$, $q_2\}$, we have 
either $p_1\not\in E$ or $p_1\not\in E_1$. We consider the case $p_1\not\in E$. 
We may assume that $p_1=E_1\cap E_2^\infty$. 
We consider the case that $F$ is a toric. In this case, $F$ corresponds to a 
primitive lattice point $v_F\in N$. Since $C$ intersects $E_1$ and $E_2^\infty$ 
transversally, we must have $j_C=1$. Thus we can assume that $k\leq 2$ by Step 4. 
One of the following holds: 
\begin{enumerate}
\renewcommand{\theenumi}{\arabic{enumi}}
\renewcommand{\labelenumi}{(\theenumi)}
\item\label{step8-1}
$v_F=(a$, $-b)$, 
\item\label{step8-2}
$v_F=(b$, $-a)$.
\end{enumerate}
For the case \eqref{step8-1}, by \cite[Corollary 7.7]{BJ}, we have 
\[
S(F)=\frac{-a(4\beta^2+9\beta+6)+b(7\beta^2+12\beta)}{3(3\beta+4)}+a(\beta+1). 
\]
Thus we have 
\[
\frac{A_{X, \Delta}(F)}{S(F)}=\frac{3(3\beta+4)(\beta+(a/b))}
{(a/b)(5\beta^2+12\beta+6)+7\beta^2+12\beta}
\geq\frac{(3\beta+4)(\beta+1)}{2(2\beta^2+4\beta+1)}>\frac{6}{5}
\]
since $a/b\in[1$, $2]$. 
For the case \eqref{step8-2}, by \cite[Corollary 7.7]{BJ}, we have 
\[
S(F)=\frac{a(7\beta^2+12\beta)-b(4\beta^2+9\beta+6)}{3(3\beta+4)}+b(\beta+1). 
\]
Thus we have 
\[
\frac{A_{X, \Delta}(F)}{S(F)}=\frac{3(3\beta+4)(\beta+(a/b))}
{(a/b)(7\beta^2+12\beta)+5\beta^2+12\beta+6}
\geq\frac{(3\beta+4)(\beta+1)}{2(2\beta^2+4\beta+1)}>\frac{6}{5}
\]
since $a/b\in[1$, $2]$. 

\noindent\underline{\textbf{Step 9}}\\
Assume that $p_1=E_1\cap E_2^\infty$ and $j_C=1$. By Step 8, we may assume that 
$k=2$ and $p_2\in F_1\setminus\left(C^{X_1}\cup E_1^{X_1}\cup(E_2^\infty)^{X_1}\right)$. 
Set 
\[
D_3:=(\beta+1)E_1+E+2\beta E_2^\infty\sim_\Q L.
\]
Then we can inductively check that 
\[
\pi^*D_3=D_3^{\tilde{X}}+\sum_{i=1}^m b_i(3\beta+1)F_i^{\tilde{X}}. 
\]
By Proposition \ref{KpM_prop}, 
$\pi^*D_3-(x/(ab))F_m^*=\nu^*(\sigma^*D_3-x F)$ is effective for 
$x\in[0$, $b(3\beta+1)]$. We have 
\begin{eqnarray*}
\left(\pi^*L-\frac{x}{ab}F_m^*\right)\cdot E_1^{\tilde{X}}&=&
\beta-\frac{x}{a}, \\
\left(\pi^*L-\frac{x}{ab}F_m^*\right)\cdot E^{\tilde{X}}&=&
\beta, \\
\left(\pi^*L-\frac{x}{ab}F_m^*\right)\cdot \left(E_2^\infty\right)^{\tilde{X}}&=&
\beta+1-\frac{x}{a}, \\
\left(\pi^*L-\frac{x}{ab}F_m^*\right)\cdot F^{\tilde{X}}&=&
\frac{x}{ab}.
\end{eqnarray*}
Note that $a\beta<b(3\beta+1)$. Thus 
$\sigma^*L-x F$ is nef for $x\in[0$, $a\beta]$, and $\sigma^*L-a\beta F$ induces 
the birational contraction $\phi\colon Y\to Y'$ of $E_1^Y$. Set 
$\nu':=\phi\circ\nu\colon\tilde{X}\to Y'$. 
By Lemma \ref{mono_lem}, we have $\left(E_1^{Y\cdot 2}\right)=-(a+b)/a$. 
Thus we have 
\begin{equation}\label{9-1_eqn}
\pi^*D_3-\frac{x}{ab}F_m^*={\nu'}^*\nu'_*\left(\pi^*D_3-\frac{x}{ab}F_m^*\right)
+\frac{x-a\beta}{a+b}\nu^*E_1^Y.
\end{equation}
For any $x\in[a\beta$, $b(3\beta+1)]$, we have 
$\nu'_*\left(\pi^*D_3-(x/(ab))F_m^*\right)\geq 0$ and 
\begin{eqnarray*}
\left(\pi^*L-\frac{x}{ab}F_m^*-\frac{x-a\beta}{a+b}\nu^*E_1^Y
\right)\cdot E^{\tilde{X}}&=&
\frac{2a\beta+b\beta-x}{a+b}, \\
\left(\pi^*L-\frac{x}{ab}F_m^*-\frac{x-a\beta}{a+b}\nu^*E_1^Y
\right)\cdot (E_2^\infty)^{\tilde{X}}&=&
\frac{2(a\beta+\frac{a+b}{2}-x)}{a+b}, \\
\left(\pi^*L-\frac{x}{ab}F_m^*-\frac{x-a\beta}{a+b}\nu^*E_1^Y
\right)\cdot F^{\tilde{X}}&=&
\frac{b\beta+x}{b(a+b)}.
\end{eqnarray*}
Note that $2a\beta+b\beta<\min\{a\beta+(a+b)/2$, $b(3\beta+1)\}$. Thus, 
for any $x\in[a\beta$, $2a\beta+b\beta]$, \eqref{9-1_eqn} gives the Zariski decomposition, 
hence
\begin{eqnarray*}
\vol_Y(\sigma^*L-x F)&=&\left(\pi^*L-\frac{x}{ab}F_m^*-
\frac{x-a\beta}{a+b}\nu^*E_1^Y\right)^{\cdot 2}\\
&=&3\beta^2+4\beta-\frac{x^2}{ab}+\frac{(x-a\beta)^2}{a(a+b)}.
\end{eqnarray*}
Moreover, $\phi_*\left(\sigma^*L-(2a\beta+b\beta)F\right)$ induces the birational 
contraction $\phi'\colon Y'\to Y''$ of $E^{Y'}$. Set 
$\nu'':=\phi'\circ\nu'\colon\tilde{X}\to Y''$. 
By Lemma \ref{mono_lem}, we have 
$(E_1^{Y\cdot 2})=-(a+b)/a$, $(E^Y\cdot E_1^Y)=1$ and $(E^{Y\cdot 2})=-1$. 
Thus we have 
\begin{equation}\label{9-2_eqn}
\pi^*D_3-\frac{x}{ab}F_m^*={\nu''}^*\nu''_*\left(\pi^*D_3-\frac{x}{ab}F_m^*\right)
+\nu^*\left(\frac{x-2a\beta}{b}E_1^Y+\frac{x-(2a\beta+b\beta)}{b}E^Y\right).
\end{equation}

For any $x\in[2a\beta+b\beta$, $b(3\beta+1)]$, we have 
$\nu''_*\left(\pi^*D_3-(x/(ab))F_m^*\right)\geq 0$ and 
\begin{eqnarray*}
&&\left(\pi^*L-\frac{x}{ab}F_m^*-\nu^*\left(\frac{x-2a\beta}{b}E_1^Y
+\frac{x-(2a\beta+b\beta)}{b}E^Y\right)\right)
\cdot (E_2^\infty)^{\tilde{X}}
=
\frac{2a\beta-b\beta+b-x}{b}, \\
&&\left(\pi^*L-\frac{x}{ab}F_m^*-\nu^*\left(\frac{x-2a\beta}{b}E_1^Y
+\frac{x-(2a\beta+b\beta)}{b}E^Y\right)\right)
\cdot F^{\tilde{X}}
=
\frac{2\beta}{b}.
\end{eqnarray*}
Note that $2a\beta-b\beta+b\leq b(3\beta+1)$. 
Thus, 
for any $x\in[2a\beta+b\beta$, $2a\beta-b\beta+b]$, 
\eqref{9-2_eqn} gives the Zariski decomposition, hence
\begin{eqnarray*}
&&\vol_Y(\sigma^*L-x F)\\
&=&
\left(\pi^*L-\frac{x}{ab}F_m^*-\nu^*\left(\frac{x-2a\beta}{b}E_1^Y
+\frac{x-(2a\beta+b\beta)}{b}E^Y\right)\right)^{\cdot 2}\\
&=&3\beta^2+4\beta-\frac{x^2}{ab}+\frac{(x-a\beta)^2}{a(a+b)}
+\frac{(x-(2a\beta+b\beta))^2}{b(a+b)}.
\end{eqnarray*}
By Proposition \ref{pltK_prop}, for $x\geq 2a\beta-b\beta+b$, 
\[
\vol_Y(\sigma^*L-x F)\leq \frac{1}{b(2b-a)}\left(x-\left(2a \beta-b\beta+b
+2(2b-a)\beta\right)\right)^2.
\]
In particular, we have 
\begin{eqnarray*}
S(F)&\leq&\frac{1}{3\beta^2+4\beta}\biggl(\int_0^{2a\beta-b\beta+b}\left(3\beta^2
+4\beta-\frac{x^2}{ab}\right)dx\\
&&+\int_{a\beta}^{2a\beta-b\beta+b}
\frac{(x-a\beta)^2}{a(a+b)}dx+\int_{2a\beta+b\beta}^{2a\beta-b\beta+b}
\frac{(x-(2a\beta+b\beta))^2}{b(a+b)}dx\\
&&+\int_{2a\beta-b\beta+b}^{2a\beta-b\beta+b+2(2b-a)\beta}
\frac{(x-(2a\beta-b\beta+b+2(2b-a)\beta))^2}{b(2b-a)}dx\biggr)\\
&=&\frac{a(-\beta^2+12\beta)+b(13\beta^2+12\beta+6)}{3(3\beta+4)}. 
\end{eqnarray*}
Thus we have 
\begin{eqnarray*}
\frac{A_{X, \Delta}(F)}{S(F)}&\geq&
\frac{3(3\beta+4)(\beta+(a/b))}{(a/b)(-\beta^2+12\beta)+13\beta^2+12\beta+6}\\
&>&\frac{(3\beta+4)(\beta+1)}{2(2\beta^2+4\beta+1)}>\frac{6}{5}
\end{eqnarray*}
since $a/b\in(1$, $2]$. 

\noindent\underline{\textbf{Step 10}}\\
Assume that $p_1=E_1\cap E_2^\infty$ and $j_C=k\geq 2$. 
For $D_3=(\beta+1)E_1+E+2\beta E_2^\infty$, 
we can inductively check that 
\[
\pi^*D_3=D_3^{\tilde{X}}+\sum_{i=1}^m b_i(3\beta+1)F_i^{\tilde{X}}. 
\]
By Proposition \ref{KpM_prop}, 
$\pi^*D_3-(x/(ab))F_m^*=\nu^*(\sigma^*D_3-x F)$ is effective for 
$x\in[0$, $b(3\beta+1)]$. 
Set 
\[
D_4:=\beta C+E_1+E\sim_\Q L. 
\]
Then we can inductively check that 
\[
\pi^*D_4=D_4^{\tilde{X}}+\sum_{i=1}^m (a_i\beta+b_i)F_i^{\tilde{X}}. 
\]
By Proposition \ref{KpM_prop}, 
$\pi^*D_4-(x/(ab))F_m^*=\nu^*(\sigma^*D_4-x F)$ is effective for 
$x\in[0$, $a\beta+b]$. 
We have 
\begin{eqnarray*}
\left(\pi^*L-\frac{x}{ab}F_m^*\right)\cdot C^{\tilde{X}}&=&
3\beta+2-\frac{x}{b}, \\
\left(\pi^*L-\frac{x}{ab}F_m^*\right)\cdot E_1^{\tilde{X}}&=&
\beta-\frac{x}{a}, \\
\left(\pi^*L-\frac{x}{ab}F_m^*\right)\cdot E^{\tilde{X}}&=&
\beta, \\
\left(\pi^*L-\frac{x}{ab}F_m^*\right)\cdot \left(E_2^\infty\right)^{\tilde{X}}&=&
\beta+1-\frac{x}{a}, \\
\left(\pi^*L-\frac{x}{ab}F_m^*\right)\cdot F^{\tilde{X}}&=&
\frac{x}{ab}.
\end{eqnarray*}

\underline{The case $a/b\geq (\beta+2)/\beta$}\,\,
We have 
$b\beta+2b\leq\min\{a\beta$, $3b\beta+2b$, $a\beta+b\}$ in this case. 
Thus 
$\sigma^*L-x F$ is nef for $x\in[0$, $b\beta+2b]$. By Proposition \ref{pltK_prop}, 
\[
S(F)\leq\frac{a(3\beta^2+4\beta)+b(\beta+2)^2}{3(\beta+2)}. 
\]
Thus we have 
\begin{eqnarray*}
\frac{A_{X, \Delta}(F)}{S(F)}&\geq&
\frac{3(\beta+2)((a/b)\beta+1)}{(a/b)(3\beta^2+4\beta)+(\beta+2)^2}\\
&\geq&\frac{3(\beta+3)}{2(2\beta+3)}>\frac{6}{5}
\end{eqnarray*}
since $a/b\in[(\beta+2)/\beta$, $\infty)$. 

\underline{The case $(3\beta+1)/\beta\leq a/b< (\beta+2)/\beta$}\,\,
We have 
$b(3\beta+1)\leq a\beta<a\beta+b$ and $a\beta<b(3\beta+2)$ in this case. 
Set 
\[
s:=\frac{b}{(a-3b)\beta}\,\in\,(0,\,1]\cap\Q
\]
and 
\[
D':=s D_3+(1-s)D_4\sim_\Q L. 
\]
Then we have 
\[
\pi^*D'={D'}^{\tilde{X}}+\sum_{i=1}^m \left(s b_i(3\beta+1)+
(1-s)(a_i\beta+b_i)\right)F_i^{\tilde{X}}.
\]
Thus $\pi^*D'-(x/(ab))F_m^*=\nu^*(\sigma^*D'-x F)$ is effective and nef for 
$x\in[0$, $a\beta]$. By Proposition \ref{pltK_prop}, 
\[
S(F)\leq\frac{a\beta+b(3\beta+4)}{3}. 
\]
Thus we have 
\begin{eqnarray*}
\frac{A_{X, \Delta}(F)}{S(F)}\geq
\frac{3((a/b)\beta+1)}{(a/b)\beta+3\beta+4}
>\frac{3(3\beta+2)}{6\beta+5}>\frac{6}{5}
\end{eqnarray*}
since $a/b\in[(3\beta+1)/\beta$, $(\beta+2)/\beta)$. 

\underline{The case $a/b< (3\beta+1)/\beta$}\,\,
We have $a\beta<b(3\beta+1)$ in this case. 
Thus 
$\sigma^*L-x F$ is nef for $x\in[0$, $a\beta]$, and $\sigma^*L-a\beta F$ induces 
the birational contraction $\phi\colon Y\to Y'$ of $E_1^Y$. Set 
$\nu':=\phi\circ\nu\colon\tilde{X}\to Y'$. 
By Lemma \ref{mono_lem}, we have $\left(E_1^{Y\cdot 2}\right)=-(a+b)/a$. 
Thus we have 
\begin{equation}\label{10-1_eqn}
\pi^*D_3-\frac{x}{ab}F_m^*={\nu'}^*\nu'_*\left(\pi^*D_3-\frac{x}{ab}F_m^*\right)
+\frac{x-a\beta}{a+b}\nu^*E_1^Y.
\end{equation}
For any $x\in[a\beta$, $b(3\beta+1)]$, we have 
$\nu'_*\left(\pi^*D_3-(x/(ab))F_m^*\right)\geq 0$ and 
\begin{eqnarray*}
\left(\pi^*L-\frac{x}{ab}F_m^*-\frac{x-a\beta}{a+b}\nu^*E_1^Y
\right)\cdot E^{\tilde{X}}&=&
\frac{2a\beta+b\beta-x}{a+b}, \\
\left(\pi^*L-\frac{x}{ab}F_m^*-\frac{x-a\beta}{a+b}\nu^*E_1^Y
\right)\cdot \left(E_2^\infty\right)^{\tilde{X}}&=&
\frac{2\left(a\beta+\frac{a+b}{2}-x\right)}{a+b}, \\
\left(\pi^*L-\frac{x}{ab}F_m^*-\frac{x-a\beta}{a+b}\nu^*E_1^Y
\right)\cdot F^{\tilde{X}}&=&
\frac{b\beta+x}{b(a+b)}.
\end{eqnarray*}
Note that $2a\beta+b\beta< a\beta+(a+b)/2$. 
Moreover, the condition $2a\beta+b\beta<b(3\beta+1)$ is equivalent to 
the condition $a/b<(2\beta+1)/(2\beta)$. 
For any $x\in[a\beta$, $\min\{2a\beta+b\beta$, $b(3\beta+1)\}]$, 
\eqref{10-1_eqn} gives the Zariski decomposition, 
hence
\begin{eqnarray*}
\vol_Y(\sigma^*L-x F)&=&\left(\pi^*L-\frac{x}{ab}F_m^*-
\frac{x-a\beta}{a+b}\nu^*E_1^Y\right)^{\cdot 2}\\
&=&3\beta^2+4\beta-\frac{x^2}{ab}+\frac{(x-a\beta)^2}{a(a+b)}.
\end{eqnarray*}

Assume firstly that $a/b\geq (2\beta+1)/(2\beta)$. By Proposition \ref{pltK_prop}, 
for $x\geq b(3\beta+1)$, 
\begin{eqnarray*}
&&\vol_Y(\sigma^*L-x F)\\ 
&\leq&\frac{(4\beta+1)^2\left(x-\left(b(3\beta+1)+
\frac{a(4\beta^2+4\beta)+b(-12\beta^2-4\beta-1)}{4\beta+1}\right)\right)^2
}{(a+b)(a(4\beta^2+4\beta)+b(-12\beta^2-4\beta-1))}.
\end{eqnarray*}
In particular, we have
\begin{eqnarray*}
S(F)&\leq &\frac{1}{3\beta^2+4\beta}\biggl(b(3\beta+1)(3\beta^2+4\beta)
-\frac{b^3(3\beta+1)^3}{3ab}\\
&&+\frac{(b(3\beta+1)-a\beta)^3}{3a(a+b)}
+\frac{(a(4\beta^2+4\beta)+b(-12\beta^2-4\beta-1))^2}{3(a+b)(4\beta+1)}\biggr)\\
&=&\frac{a(12\beta^3+31\beta^2+16\beta)+b(36\beta^3+69\beta^2+40\beta+4)}{
3(4\beta+1)(3\beta+4)}.
\end{eqnarray*}
Thus we have 
\begin{eqnarray*}
\frac{A_{X, \Delta}(F)}{S(F)}&\geq&
\frac{3(4\beta+1)(3\beta+4)((a/b)\beta+1)}{
(a/b)(12\beta^3+31\beta^2+16\beta)+36\beta^3+69\beta^2+40\beta+4}\\
&>&\frac{3(3\beta+2)}{6\beta+5}>\frac{6}{5}
\end{eqnarray*}
since $a/b\in[(2\beta+1)/(2\beta)$, $(3\beta+1)/\beta)$. 

Thus we may assume that $a/b<(2\beta+1)/(2\beta)$. In this case, 
$\phi_*\left(\sigma^*L-(2a\beta+b\beta)F\right)$ induces the birational 
contraction $\phi'\colon Y'\to Y''$ of $E^{Y'}$. Set 
$\nu'':=\phi'\circ\nu'\colon\tilde{X}\to Y''$. 
By Lemma \ref{mono_lem}, we have 
$\left(E_1^{Y\cdot 2}\right)=-(a+b)/a$, $\left(E^Y\cdot E_1^Y\right)=1$ 
and $\left(E^{Y\cdot 2}\right)=-1$. 
Thus we have 
\begin{equation}\label{10-2_eqn}
\pi^*D_3-\frac{x}{ab}F_m^*={\nu''}^*\nu''_*\left(\pi^*D_3-\frac{x}{ab}F_m^*\right)
+\nu^*\left(\frac{x-2a\beta}{b}E_1^Y+\frac{x-(2a\beta+b\beta)}{b}E^Y\right).
\end{equation}
For any $x\in[2a\beta+b\beta$, $b(3\beta+1)]$, we have 
$\nu''_*\left(\pi^*D_3-(x/(ab))F_m^*\right)\geq 0$ and 
\begin{eqnarray*}
\left(\pi^*L-\frac{x}{ab}F_m^*-\nu^*\left(\frac{x-2a\beta}{b}E_1^Y
+\frac{x-(2a\beta+b\beta)}{b}E^Y\right)\right)
\cdot (E_2^\infty)^{\tilde{X}}
&=&
\frac{2a\beta-b\beta+b-x}{b}, \\
\left(\pi^*L-\frac{x}{ab}F_m^*-\nu^*\left(\frac{x-2a\beta}{b}E_1^Y
+\frac{x-(2a\beta+b\beta)}{b}E^Y\right)\right)
\cdot F^{\tilde{X}}
&=&
\frac{2\beta}{b}.
\end{eqnarray*}
Thus, 
for any $x\in[2a\beta+b\beta$, $\min\{2a\beta-b\beta+b$, $b(3\beta+1)\}]$, 
\eqref{10-2_eqn} gives the Zariski decomposition, hence
\begin{eqnarray*}
&&\vol_Y(\sigma^*L-x F)\\
&=&
\left(\pi^*L-\frac{x}{ab}F_m^*-\nu^*\left(\frac{x-2a\beta}{b}E_1^Y
+\frac{x-(2a\beta+b\beta)}{b}E^Y\right)\right)^{\cdot 2}\\
&=&3\beta^2+4\beta-\frac{x^2}{ab}+\frac{(x-a\beta)^2}{a(a+b)}
+\frac{(x-(2a\beta+b\beta))^2}{b(a+b)}.
\end{eqnarray*}

Assume that $a/b\geq 2$. We have $b(3\beta+1)\leq 2a\beta-b\beta+b$ in this case. 
By Proposition \ref{pltK_prop}, for $x\geq b(3\beta+1)$, 
\[
\vol_Y(\sigma^*L-x F)\leq \frac{1}{b(a-2b)}\left(x-\left(b(3\beta+1)
+2(a-2b)\beta\right)\right)^2.
\]
In particular, we have 
\begin{eqnarray*}
S(F)&\leq&\frac{1}{3\beta^2+4\beta}\biggl(\int_0^{b(3\beta+1)}\left(3\beta^2
+4\beta-\frac{x^2}{ab}\right)dx\\
&&+\int_{a\beta}^{b(3\beta+1)}
\frac{(x-a\beta)^2}{a(a+b)}dx+\int_{2a\beta+b\beta}^{b(3\beta+1)}
\frac{(x-(2a\beta+b\beta))^2}{b(a+b)}dx\\
&&+\int_{b(3\beta+1)}^{b(3\beta+1)+2(a-2b)\beta}
\frac{(x-(b(3\beta+1)+2(a-2b)\beta))^2}{b(a-2b)}dx\biggr)\\
&=&\frac{a(-\beta^2+12\beta)+b(13\beta^2+12\beta+6)}{3(3\beta+4)}. 
\end{eqnarray*}
Thus we have 
\begin{eqnarray*}
\frac{A_{X, \Delta}(F)}{S(F)}&\geq&
\frac{3(3\beta+4)((a/b)\beta+1)}{(a/b)(-\beta^2+12\beta)+13\beta^2+12\beta+6}\\
&>&\frac{3(3\beta+4)(2\beta+3)}{24\beta^2+47\beta+24}>\frac{6}{5}
\end{eqnarray*}
since $a/b\in[2$, $(2\beta+1)/(2\beta))$. 

Therefore we may assume that $a/b<2$. We have $2a\beta-b\beta+b<b(3\beta+1)$ 
in this case. 
By Proposition \ref{pltK_prop}, for $x\geq 2a\beta-b\beta+b$, 
\[
\vol_Y(\sigma^*L-x F)\leq \frac{1}{b(2b-a)}\left(x-\left(2a\beta-b\beta+b
+2(2b-a)\beta\right)\right)^2.
\]
In particular, we have 
\begin{eqnarray*}
S(F)&\leq&\frac{1}{3\beta^2+4\beta}\biggl(\int_0^{2a\beta-b\beta+b}\left(3\beta^2
+4\beta-\frac{x^2}{ab}\right)dx\\
&&+\int_{a\beta}^{2a\beta-b\beta+b}
\frac{(x-a\beta)^2}{a(a+b)}dx+\int_{2a\beta+b\beta}^{2a\beta-b\beta+b}
\frac{(x-(2a\beta+b\beta))^2}{b(a+b)}dx\\
&&+\int_{2a\beta-b\beta+b}^{2a\beta-b\beta+b+2(2b-a)\beta}
\frac{(x-(2a\beta-b\beta+b+2(2b-a)\beta))^2}{b(2b-a)}dx\biggr)\\
&=&\frac{a(-\beta^2+12\beta)+b(13\beta^2+12\beta+6)}{3(3\beta+4)}. 
\end{eqnarray*}
Thus we have 
\begin{eqnarray*}
\frac{A_{X, \Delta}(F)}{S(F)}&\geq&
\frac{3(3\beta+4)((a/b)\beta+1)}{(a/b)(-\beta^2+12\beta)+13\beta^2+12\beta+6}\\
&>&\frac{3(3\beta+4)(2\beta+1)}{11\beta^2+36\beta+6}>\frac{6}{5}
\end{eqnarray*}
since $a/b\in(1$, $2)$. 

\noindent\underline{\textbf{Step 11}}\\
Assume that $p_1\in C\cap E$. As we have seen in Step 8, we have $p_1\not\in E_1$. 
Of course, $p_1\not\in E_2$ since $C\cap E_2=\emptyset$. 
Assume that $j_C=1$. Moreover, we assume that $p_2\not\in E^{X_1}$ if $m\geq 2$. 
By Step 5, we can assume that $k\leq 2$. For 
$D=(\beta+1)E_1+2\beta E_2+(2\beta+1)E$, we can 
inductively check that 
\[
\pi^*D=D^{\tilde{X}}+\sum_{i=1}^m b_i(2\beta+1)F_i^{\tilde{X}}. 
\]
By Proposition \ref{KpM_prop}, 
$\pi^*D-(x/(ab))F_m^*=\nu^*(\sigma^*D-x F)$ is effective for 
$x\in[0$, $b(2\beta+1)]$. We have 
\begin{eqnarray*}
\left(\pi^*L-\frac{x}{ab}F_m^*\right)\cdot E_1^{\tilde{X}}&=&
\beta, \\
\left(\pi^*L-\frac{x}{ab}F_m^*\right)\cdot E_2^{\tilde{X}}&=&
1, \\
\left(\pi^*L-\frac{x}{ab}F_m^*\right)\cdot E^{\tilde{X}}&=&
\beta-\frac{x}{a}, \\
\left(\pi^*L-\frac{x}{ab}F_m^*\right)\cdot F^{\tilde{X}}&=&
\frac{x}{ab}.
\end{eqnarray*}
Note that $a\beta<b(2\beta+1)$. Thus 
$\sigma^*L-x F$ is nef for $x\in[0$, $a\beta]$, and $\sigma^*L-a\beta F$ induces 
the birational contraction $\phi\colon Y\to Y'$ of $E^Y$. Set 
$\nu':=\phi\circ\nu\colon\tilde{X}\to Y'$. 
By Lemma \ref{mono_lem}, we have $\left(E^{Y\cdot 2}\right)=-(a+b)/a$. 
Thus we have 
\begin{equation}\label{11-1_eqn}
\pi^*D-\frac{x}{ab}F_m^*={\nu'}^*\nu'_*\left(\pi^*D-\frac{x}{ab}F_m^*\right)
+\frac{x-a\beta}{a+b}\nu^*E^Y.
\end{equation}
For any $x\in[a\beta$, $b(2\beta+1)]$, we have 
$\nu'_*\left(\pi^*D-(x/(ab))F_m^*\right)\geq 0$ and 
\begin{eqnarray*}
\left(\pi^*L-\frac{x}{ab}F_m^*-\frac{x-a\beta}{a+b}\nu^*E^Y
\right)\cdot E_1^{\tilde{X}}&=&
\frac{2a\beta+b\beta-x}{a+b}, \\
\left(\pi^*L-\frac{x}{ab}F_m^*-\frac{x-a\beta}{a+b}\nu^*E^Y
\right)\cdot E_2^{\tilde{X}}&=&
\frac{a\beta+a+b-x}{a+b}, \\
\left(\pi^*L-\frac{x}{ab}F_m^*-\frac{x-a\beta}{a+b}\nu^*E^Y
\right)\cdot F^{\tilde{X}}&=&
\frac{b\beta+x}{b(a+b)}.
\end{eqnarray*}
Note that $2a\beta+b\beta<\min\{a\beta+a+b$, $b(2\beta+1)\}$. Thus, 
for any $x\in[a\beta$, $2a\beta+b\beta]$, \eqref{11-1_eqn} 
gives the Zariski decomposition, hence
\begin{eqnarray*}
\vol_Y(\sigma^*L-x F)&=&\left(\pi^*L-\frac{x}{ab}F_m^*-
\frac{x-a\beta}{a+b}\nu^*E^Y\right)^{\cdot 2}\\
&=&3\beta^2+4\beta-\frac{x^2}{ab}+\frac{(x-a\beta)^2}{a(a+b)}.
\end{eqnarray*}
Moreover, $\phi_*\left(\sigma^*L-(2a\beta+b\beta)F\right)$ induces the birational 
contraction $\phi'\colon Y'\to Y''$ of $E_1^{Y'}$. Set 
$\nu'':=\phi'\circ\nu'\colon\tilde{X}\to Y''$. 
By Lemma \ref{mono_lem}, we have 
$\left(E^{Y\cdot 2}\right)=-(a+b)/a$, $\left(E^Y\cdot E_1^Y\right)=1$ 
and $\left(E_1^{Y\cdot 2}\right)=-1$. 
Thus we have 
\begin{equation}\label{11-2_eqn}
\pi^*D-\frac{x}{ab}F_m^*={\nu''}^*\nu''_*\left(\pi^*D-\frac{x}{ab}F_m^*\right)
+\nu^*\left(\frac{x-2a\beta}{b}E^Y+\frac{x-(2a\beta+b\beta)}{b}E_1^Y\right).
\end{equation}

For any $x\in[2a\beta+b\beta$, $b(2\beta+1)]$, we have 
$\nu''_*\left(\pi^*D-(x/(ab))F_m^*\right)\geq 0$ and 
\begin{eqnarray*}
\left(\pi^*L-\frac{x}{ab}F_m^*-\nu^*\left(\frac{x-2a\beta}{b}E^Y
+\frac{x-(2a\beta+b\beta)}{b}E_1^Y\right)\right)
\cdot E_2^{\tilde{X}}
&=&
\frac{2a\beta+b-x}{b}, \\
\left(\pi^*L-\frac{x}{ab}F_m^*-\nu^*\left(\frac{x-2a\beta}{b}E^Y
+\frac{x-(2a\beta+b\beta)}{b}E_1^Y\right)\right)
\cdot F^{\tilde{X}}
&=&
\frac{2\beta}{b}.
\end{eqnarray*}
Note that $b(2\beta+1)<2a\beta+b$. 
Thus, 
for any $x\in[2a\beta+b\beta$, $b(2\beta+1)]$, 
\eqref{11-2_eqn} gives the Zariski decomposition, hence
\begin{eqnarray*}
&&\vol_Y(\sigma^*L-x F)\\
&=&
\left(\pi^*L-\frac{x}{ab}F_m^*-\nu^*\left(\frac{x-2a\beta}{b}E^Y
+\frac{x-(2a\beta+b\beta)}{b}E_1^Y\right)\right)^{\cdot 2}\\
&=&3\beta^2+4\beta-\frac{x^2}{ab}+\frac{(x-a\beta)^2}{a(a+b)}
+\frac{(x-(2a\beta+b\beta))^2}{b(a+b)}.
\end{eqnarray*}
By Proposition \ref{pltK_prop}, for $x\geq b(2\beta+1)$, 
\[
\vol_Y(\sigma^*L-x F)\leq \frac{1}{b(a-b)}\left(x-\left(b(2\beta+1)
+2(a-b)\beta\right)\right)^2.
\]
In particular, we have 
\begin{eqnarray*}
S(F)&\leq&\frac{1}{3\beta^2+4\beta}\biggl(\int_0^{b(2\beta+1)}\left(3\beta^2
+4\beta-\frac{x^2}{ab}\right)dx\\
&&+\int_{a\beta}^{b(2\beta+1)}
\frac{(x-a\beta)^2}{a(a+b)}dx+\int_{2a\beta+b\beta}^{b(2\beta+1)}
\frac{(x-(2a\beta+b\beta))^2}{b(a+b)}dx\\
&&+\int_{b(2\beta+1)}^{b(2\beta+1)+2(a-b)\beta}
\frac{(x-(b(2\beta+1)+2(a-b)\beta))^2}{b(a-b)}dx\biggr)\\
&=&\frac{a(3\beta^2+12\beta)+b(7\beta^2+12\beta+6)}{3(3\beta+4)}. 
\end{eqnarray*}
Thus we have 
\begin{eqnarray*}
\frac{A_{X, \Delta}(F)}{S(F)}&\geq&
\frac{3(3\beta+4)(\beta+(a/b))}{(a/b)(3\beta^2+12\beta)+7\beta^2+12\beta+6}\\
&\geq&\frac{3(3\beta+4)(\beta+1)}{2(5\beta^2+12\beta+3)}>\frac{6}{5}
\end{eqnarray*}
since $a/b\in[1$, $2]$. 

\noindent\underline{\textbf{Step 12}}\\
Assume that $p_1\in C\cap E$ and $j_C=1$. By Step 11, we can assume that 
$k=2$ and $p_2\in E^{X_1}$. For 
$D=(\beta+1)E_1+2\beta E_2+(2\beta+1)E$, we can 
inductively check that 
\[
\pi^*D=D^{\tilde{X}}+\sum_{i=1}^m a_i(2\beta+1)F_i^{\tilde{X}}. 
\]
By Proposition \ref{KpM_prop}, 
$\pi^*D-(x/(ab))F_m^*=\nu^*(\sigma^*D-x F)$ is effective for 
$x\in[0$, $a(2\beta+1)]$. We have 
\begin{eqnarray*}
\left(\pi^*L-\frac{x}{ab}F_m^*\right)\cdot E_1^{\tilde{X}}&=&
\beta, \\
\left(\pi^*L-\frac{x}{ab}F_m^*\right)\cdot E_2^{\tilde{X}}&=&
1, \\
\left(\pi^*L-\frac{x}{ab}F_m^*\right)\cdot E^{\tilde{X}}&=&
\beta-\frac{x}{b}, \\
\left(\pi^*L-\frac{x}{ab}F_m^*\right)\cdot F^{\tilde{X}}&=&
\frac{x}{ab}.
\end{eqnarray*}
Note that $b\beta<a(2\beta+1)$. Thus 
$\sigma^*L-x F$ is nef for $x\in[0$, $b\beta]$, and $\sigma^*L-b\beta F$ induces 
the birational contraction $\phi\colon Y\to Y'$ of $E^Y$. Set 
$\nu':=\phi\circ\nu\colon\tilde{X}\to Y'$. 
By Lemma \ref{mono_lem}, we have $\left(E^{Y\cdot 2}\right)=-(a+b)/b$. 
Thus we have 
\begin{equation}\label{12-1_eqn}
\pi^*D-\frac{x}{ab}F_m^*={\nu'}^*\nu'_*\left(\pi^*D-\frac{x}{ab}F_m^*\right)
+\frac{x-b\beta}{a+b}\nu^*E^Y.
\end{equation}
For any $x\in[b\beta$, $a(2\beta+1)]$, we have 
$\nu'_*\left(\pi^*D-(x/(ab))F_m^*\right)\geq 0$ and 
\begin{eqnarray*}
\left(\pi^*L-\frac{x}{ab}F_m^*-\frac{x-b\beta}{a+b}\nu^*E^Y
\right)\cdot E_1^{\tilde{X}}&=&
\frac{a\beta+2b\beta-x}{a+b}, \\
\left(\pi^*L-\frac{x}{ab}F_m^*-\frac{x-b\beta}{a+b}\nu^*E^Y
\right)\cdot E_2^{\tilde{X}}&=&
\frac{b\beta+a+b-x}{a+b}, \\
\left(\pi^*L-\frac{x}{ab}F_m^*-\frac{x-b\beta}{a+b}\nu^*E^Y
\right)\cdot F^{\tilde{X}}&=&
\frac{a\beta+x}{a(a+b)}.
\end{eqnarray*}
Note that $a\beta+2b\beta<\min\{b\beta+a+b$, $a(2\beta+1)\}$. Thus, 
for any $x\in[a\beta$, $a\beta+2b\beta]$, \eqref{12-1_eqn} 
gives the Zariski decomposition, hence
\begin{eqnarray*}
\vol_Y(\sigma^*L-x F)&=&\left(\pi^*L-\frac{x}{ab}F_m^*-
\frac{x-b\beta}{a+b}\nu^*E^Y\right)^{\cdot 2}\\
&=&3\beta^2+4\beta-\frac{x^2}{ab}+\frac{(x-b\beta)^2}{b(a+b)}.
\end{eqnarray*}
Moreover, $\phi_*\left(\sigma^*L-(a\beta+2b\beta)F\right)$ induces the birational 
contraction $\phi'\colon Y'\to Y''$ of $E_1^{Y'}$. Set 
$\nu'':=\phi'\circ\nu'\colon\tilde{X}\to Y''$. 
By Lemma \ref{mono_lem}, we have 
$\left(E^{Y\cdot 2}\right)=-(a+b)/b$, $\left(E^Y\cdot E_1^Y\right)=1$ 
and $\left(E_1^{Y\cdot 2}\right)=-1$. 
Thus we have 
\begin{equation}\label{12-2_eqn}
\pi^*D-\frac{x}{ab}F_m^*={\nu''}^*\nu''_*\left(\pi^*D-\frac{x}{ab}F_m^*\right)
+\nu^*\left(\frac{x-2b\beta}{a}E^Y+\frac{x-(a\beta+2b\beta)}{a}E_1^Y\right).
\end{equation}

For any $x\in[a\beta+2b\beta$, $a(2\beta+1)]$, we have 
$\nu''_*\left(\pi^*D-(x/(ab))F_m^*\right)\geq 0$ and 
\begin{eqnarray*}
&&\left(\pi^*L-\frac{x}{ab}F_m^*-\nu^*\left(\frac{x-2b\beta}{a}E^Y
+\frac{x-(a\beta+2b\beta)}{a}E_1^Y\right)\right)
\cdot E_2^{\tilde{X}}
=
\frac{2b\beta+a-x}{a}, \\
&&\left(\pi^*L-\frac{x}{ab}F_m^*-\nu^*\left(\frac{x-2b\beta}{a}E^Y
+\frac{x-(a\beta+2b\beta)}{a}E_1^Y\right)\right)
\cdot F^{\tilde{X}}
=
\frac{2\beta}{a}.
\end{eqnarray*}
Note that $2b\beta+a<a(2\beta+1)$. 
Thus, 
for any $x\in[a\beta+2b\beta$, $2b\beta+a]$, 
\eqref{12-2_eqn} gives the Zariski decomposition, hence
\begin{eqnarray*}
&&\vol_Y(\sigma^*L-x F)\\
&=&
\left(\pi^*L-\frac{x}{ab}F_m^*-\nu^*\left(\frac{x-2b\beta}{a}E^Y
+\frac{x-(a\beta+2b\beta)}{a}E_1^Y\right)\right)^{\cdot 2}\\
&=&3\beta^2+4\beta-\frac{x^2}{ab}+\frac{(x-b\beta)^2}{b(a+b)}
+\frac{(x-(a\beta+2b\beta))^2}{a(a+b)}.
\end{eqnarray*}
By Proposition \ref{pltK_prop}, for $x\geq 2b\beta+a$, 
\[
\vol_Y(\sigma^*L-x F)\leq \frac{1}{a(a-b)}\left(x-\left(2b\beta+a
+2(a-b)\beta\right)\right)^2.
\]
In particular, we have 
\begin{eqnarray*}
S(F)&\leq&\frac{1}{3\beta^2+4\beta}\biggl(\int_0^{2b\beta+a}\left(3\beta^2
+4\beta-\frac{x^2}{ab}\right)dx\\
&&+\int_{b\beta}^{2b\beta+a}
\frac{(x-b\beta)^2}{b(a+b)}dx+\int_{a\beta+2b\beta}^{2b\beta+a}
\frac{(x-(a\beta+2b\beta))^2}{a(a+b)}dx\\
&&+\int_{2b\beta+a}^{2b\beta+a+2(a-b)\beta}
\frac{(x-(2b\beta+a+2(a-b)\beta))^2}{a(a-b)}dx\biggr)\\
&=&\frac{a(7\beta^2+12\beta+6)+b(3\beta^2+12\beta)}{3(3\beta+4)}. 
\end{eqnarray*}
Thus we have 
\begin{eqnarray*}
\frac{A_{X, \Delta}(F)}{S(F)}&\geq&
\frac{3(3\beta+4)(\beta+(a/b))}{(a/b)(7\beta^2+12\beta+6)+3\beta^2+12\beta}\\
&>&\frac{3(3\beta+4)(\beta+1)}{2(5\beta^2+12\beta+3)}>\frac{6}{5}
\end{eqnarray*}
since $a/b\in(1$, $2]$. 

\noindent\underline{\textbf{Step 13}}\\
Assume that $p_1\in C\cap E$. By Steps 11 and 12, we may assume that 
$j_C=k\geq 2$. 
For $D=(\beta+1)E_1+2\beta E_2+(2\beta+1)E$, 
we can inductively check that 
\[
\pi^*D=D^{\tilde{X}}+\sum_{i=1}^m b_i(2\beta+1)F_i^{\tilde{X}}. 
\]
By Proposition \ref{KpM_prop}, 
$\pi^*D-(x/(ab))F_m^*=\nu^*(\sigma^*D-x F)$ is effective for 
$x\in[0$, $b(2\beta+1)]$. 
For $D_4=\beta C+E_1+E$, 
we can inductively check that 
\[
\pi^*D_4=D_4^{\tilde{X}}+\sum_{i=1}^m (a_i\beta+b_i)F_i^{\tilde{X}}. 
\]
By Proposition \ref{KpM_prop}, 
$\pi^*D_4-(x/(ab))F_m^*=\nu^*(\sigma^*D_4-x F)$ is effective for 
$x\in[0$, $a\beta+b]$. 
We have 
\begin{eqnarray*}
\left(\pi^*L-\frac{x}{ab}F_m^*\right)\cdot C^{\tilde{X}}&=&
3\beta+2-\frac{x}{b}, \\
\left(\pi^*L-\frac{x}{ab}F_m^*\right)\cdot E_1^{\tilde{X}}&=&
\beta, \\
\left(\pi^*L-\frac{x}{ab}F_m^*\right)\cdot E_2^{\tilde{X}}&=&
1, \\
\left(\pi^*L-\frac{x}{ab}F_m^*\right)\cdot E^{\tilde{X}}&=&
\beta-\frac{x}{a}, \\
\left(\pi^*L-\frac{x}{ab}F_m^*\right)\cdot F^{\tilde{X}}&=&
\frac{x}{ab}.
\end{eqnarray*}

\underline{The case $a/b\geq (3\beta+2)/\beta$}\,\,
We have 
$3b\beta+2b\leq a\beta<a\beta+b$ in this case. 
Thus 
$\sigma^*L-x F$ is nef for $x\in[0$, $3b\beta+2b]$. By Proposition \ref{pltK_prop}, 
\[
S(F)\leq\frac{a(3\beta^2+4\beta)+b(3\beta+2)^2}{3(3\beta+2)}. 
\]
Thus we have 
\begin{eqnarray*}
\frac{A_{X, \Delta}(F)}{S(F)}\geq
\frac{3(3\beta+2)((a/b)\beta+1)}{(a/b)(3\beta^2+4\beta)+(3\beta+2)^2}
\geq\frac{3}{2}>\frac{6}{5}
\end{eqnarray*}
since $a/b\in[(3\beta+2)/\beta$, $\infty)$. 

\underline{The case $(2\beta+1)/\beta\leq a/b< (3\beta+2)/\beta$}\,\,
We have 
$b(2\beta+1)\leq a\beta<b(3\beta+2)$ in this case. 
Set 
\[
t:=\frac{a\beta-2b\beta-b}{(a-2b)\beta}\,\in\,[0,\,1)\cap\Q
\]
and 
\[
D_t:=tD_4+(1-t)D\sim_\Q L. 
\]
Then we have 
\[
\pi^*D_t=D_t^{\tilde{X}}+\sum_{i=1}^m \left(t (a_i\beta+b_i)+
(1-t)b_i(2\beta+1)\right)F_i^{\tilde{X}}.
\]
Thus $\pi^*D_t-(x/(ab))F_m^*=\nu^*(\sigma^*D_t-x F)$ is effective and nef for 
$x\in[0$, $a\beta]$. By Proposition \ref{pltK_prop}, 
\[
S(F)\leq\frac{a\beta+b(3\beta+4)}{3}. 
\]
Thus we have 
\begin{eqnarray*}
\frac{A_{X, \Delta}(F)}{S(F)}\geq
\frac{3((a/b)\beta+1)}{(a/b)\beta+3\beta+4}
\geq\frac{6}{5}
\end{eqnarray*}
since $a/b\in[(2\beta+1)/\beta$, $(3\beta+2)/\beta)$. 

\underline{The case $a/b< (2\beta+1)/\beta$}\,\,
We have $a\beta<b(2\beta+1)$ in this case. 
Thus 
$\sigma^*L-x F$ is nef for $x\in[0$, $a\beta]$, and $\sigma^*L-a\beta F$ induces 
the birational contraction $\phi\colon Y\to Y'$ of $E^Y$. Set 
$\nu':=\phi\circ\nu\colon\tilde{X}\to Y'$. 
By Lemma \ref{mono_lem}, we have $\left(E^{Y\cdot 2}\right)=-(a+b)/a$. 
Thus we have 
\begin{equation}\label{13-1_eqn}
\pi^*D-\frac{x}{ab}F_m^*={\nu'}^*\nu'_*\left(\pi^*D-\frac{x}{ab}F_m^*\right)
+\frac{x-a\beta}{a+b}\nu^*E^Y.
\end{equation}
For any $x\in[a\beta$, $b(2\beta+1)]$, we have 
$\nu'_*\left(\pi^*D-(x/(ab))F_m^*\right)\geq 0$ and 
\begin{eqnarray*}
\left(\pi^*L-\frac{x}{ab}F_m^*-\frac{x-a\beta}{a+b}\nu^*E^Y
\right)\cdot E_1^{\tilde{X}}&=&
\frac{2a\beta+b\beta-x}{a+b}, \\
\left(\pi^*L-\frac{x}{ab}F_m^*-\frac{x-a\beta}{a+b}\nu^*E^Y
\right)\cdot E_2^{\tilde{X}}&=&
\frac{a\beta+a+b-x}{a+b}, \\
\left(\pi^*L-\frac{x}{ab}F_m^*-\frac{x-a\beta}{a+b}\nu^*E^Y
\right)\cdot F^{\tilde{X}}&=&
\frac{b\beta+x}{b(a+b)}.
\end{eqnarray*}
Note that the condition $2a\beta+b\beta<b(2\beta+1)$ is equivalent to 
the condition $a/b<(\beta+1)/(2\beta)$. 
For any $x\in[a\beta$, $\min\{2a\beta+b\beta$, $b(2\beta+1)\}]$, 
\eqref{13-1_eqn} gives the Zariski decomposition, 
hence
\begin{eqnarray*}
\vol_Y(\sigma^*L-x F)&=&\left(\pi^*L-\frac{x}{ab}F_m^*-
\frac{x-a\beta}{a+b}\nu^*E^Y\right)^{\cdot 2}\\
&=&3\beta^2+4\beta-\frac{x^2}{ab}+\frac{(x-a\beta)^2}{a(a+b)}.
\end{eqnarray*}

Assume firstly that $a/b\geq (\beta+1)/(2\beta)$. By Proposition \ref{pltK_prop}, 
for $x\geq b(2\beta+1)$, 
\begin{eqnarray*}
&&\vol_Y(\sigma^*L-x F)\\ 
&\leq&\frac{(3\beta+1)^2\left(x-\left(b(2\beta+1)+
\frac{a(4\beta^2+4\beta)+b(-5\beta^2-2\beta-1)}{3\beta+1}\right)\right)^2
}{(a+b)(a(4\beta^2+4\beta)+b(-5\beta^2-2\beta-1))}.
\end{eqnarray*}
In particular, we have
\begin{eqnarray*}
S(F)&\leq &\frac{1}{3\beta^2+4\beta}\biggl(b(2\beta+1)(3\beta^2+4\beta)
-\frac{b^3(2\beta+1)^3}{3ab}\\
&&+\frac{(b(2\beta+1)-a\beta)^3}{3a(a+b)}
+\frac{(a(4\beta^2+4\beta)+b(-5\beta^2-2\beta-1))^2}{3(a+b)(3\beta+1)}\biggr)\\
&=&\frac{a(13\beta^3+31\beta^2+16\beta)+b(19\beta^3+45\beta^2+32\beta+4)}{
3(3\beta+1)(3\beta+4)}.
\end{eqnarray*}
Thus we have 
\begin{eqnarray*}
\frac{A_{X, \Delta}(F)}{S(F)}&\geq&
\frac{3(3\beta+1)(3\beta+4)((a/b)\beta+1)}{
(a/b)(13\beta^3+31\beta^2+16\beta)+19\beta^3+45\beta^2+32\beta+4}
>\frac{6}{5}
\end{eqnarray*}
since $a/b\in[(\beta+1)/(2\beta)$, $(2\beta+1)/\beta)$. 

Thus we may assume that $a/b<(\beta+1)/(2\beta)$. In this case, 
$\phi_*\left(\sigma^*L-(2a\beta+b\beta)F\right)$ induces the birational 
contraction $\phi'\colon Y'\to Y''$ of $E_1^{Y'}$. Set 
$\nu'':=\phi'\circ\nu'\colon\tilde{X}\to Y''$. 
By Lemma \ref{mono_lem}, we have 
$\left(E^{Y\cdot 2}\right)=-(a+b)/a$, $\left(E^Y\cdot E_1^Y\right)=1$ 
and $\left(E_1^{Y\cdot 2}\right)=-1$. 
Thus we have 
\begin{equation}\label{13-2_eqn}
\pi^*D-\frac{x}{ab}F_m^*={\nu''}^*\nu''_*\left(\pi^*D-\frac{x}{ab}F_m^*\right)
+\nu^*\left(\frac{x-2a\beta}{b}E^Y+\frac{x-(2a\beta+b\beta)}{b}E_1^Y\right).
\end{equation}
For any $x\in[2a\beta+b\beta$, $b(2\beta+1)]$, we have 
$\nu''_*\left(\pi^*D-(x/(ab))F_m^*\right)\geq 0$ and 
\begin{eqnarray*}
&&\left(\pi^*L-\frac{x}{ab}F_m^*-\nu^*\left(\frac{x-2a\beta}{b}E^Y
+\frac{x-(2a\beta+b\beta)}{b}E_1^Y\right)\right)
\cdot E_2^{\tilde{X}}
=
\frac{2a\beta+b-x}{b}, \\
&&\left(\pi^*L-\frac{x}{ab}F_m^*-\nu^*\left(\frac{x-2a\beta}{b}E^Y
+\frac{x-(2a\beta+b\beta)}{b}E_1^Y\right)\right)
\cdot F^{\tilde{X}}
=
\frac{2\beta}{b}.
\end{eqnarray*}
Thus, 
for any $x\in[2a\beta+b\beta$, $b(2\beta+1)]$, 
\eqref{13-2_eqn} gives the Zariski decomposition, hence
\begin{eqnarray*}
&&\vol_Y(\sigma^*L-x F)\\
&=&
\left(\pi^*L-\frac{x}{ab}F_m^*-\nu^*\left(\frac{x-2a\beta}{b}E^Y
+\frac{x-(2a\beta+b\beta)}{b}E_1^Y\right)\right)^{\cdot 2}\\
&=&3\beta^2+4\beta-\frac{x^2}{ab}+\frac{(x-a\beta)^2}{a(a+b)}
+\frac{(x-(2a\beta+b\beta))^2}{b(a+b)}.
\end{eqnarray*}
By Proposition \ref{pltK_prop}, for $x\geq b(2\beta+1)$, 
\[
\vol_Y(\sigma^*L-x F)\leq \frac{1}{b(a-b)}\left(x-\left(b(2\beta+1)
+2(a-b)\beta\right)\right)^2.
\]
In particular, we have 
\begin{eqnarray*}
S(F)&\leq&\frac{1}{3\beta^2+4\beta}\biggl(\int_0^{b(2\beta+1)}\left(3\beta^2
+4\beta-\frac{x^2}{ab}\right)dx\\
&&+\int_{a\beta}^{b(2\beta+1)}
\frac{(x-a\beta)^2}{a(a+b)}dx+\int_{2a\beta+b\beta}^{b(2\beta+1)}
\frac{(x-(2a\beta+b\beta))^2}{b(a+b)}dx\\
&&+\int_{b(2\beta+1)}^{b(2\beta+1)+2(a-b)\beta}
\frac{(x-(b(2\beta+1)+2(a-b)\beta))^2}{b(a-b)}dx\biggr)\\
&=&\frac{a(3\beta^2+12\beta)+b(7\beta^2+12\beta+6)}{3(3\beta+4)}. 
\end{eqnarray*}
Thus we have 
\begin{eqnarray*}
\frac{A_{X, \Delta}(F)}{S(F)}&\geq&
\frac{3(3\beta+4)((a/b)\beta+1)}{(a/b)(3\beta^2+12\beta)+7\beta^2+12\beta+6}\\
&>&\frac{3(3\beta+4)(\beta+3)}{17\beta^2+39\beta+24}>\frac{6}{5}
\end{eqnarray*}
since $a/b\in(1$, $(\beta+1)/(2\beta))$. Thus we have completed the proof. 
\end{proof}

\section{Prime divisors centered at special points, I}\label{caseI_section}

In this section, we prove the following: 

\begin{thm}\label{caseI_thm}
Let $(X, C)$ be an asymptotically log del Pezzo surface of type 
$(\operatorname{I.9B.}n)$, let $\eta\colon X\to\pr^1$ be the anti-log-canonical 
morphism, and let $q_1$, $q_2\in C$ be the ramification points of 
$\eta|_C\colon C\to\pr^1$. Take any exceptional prime divisor $F$ over $X$ such that 
$c_X(F)=q_1$. Assume that $\eta^{-1}(\eta(q_1))$ is smooth. Then, for any 
$\beta\in(0$, $1/(7n))\cap\Q$, we have 
\[
\frac{A_{X, \Delta}(F)}{S_L(F)}\geq 
\frac{3((4-n)\beta+4)(2\beta+1)}{(n^2-10n+24)\beta^2
+(-6n+36)\beta+12},
\]
where $\Delta:=(1-\beta)C$ and $L:=-(K_X+\Delta)$. Moreover, the above 
inequality is optimal; there uniquely exists a prime divisor over $X$ satisfying the above 
conditions such that equality holds. 
\end{thm}

\begin{proof}
The following proof is divided into 4 numbers of steps. 

\noindent\underline{\textbf{Step 1}}\\
By Proposition \ref{pltK_prop}, we may assume that $F$ is dreamy and plt-type over 
$(X, \Delta)$. We follow the notations in \S \ref{plt_section}. Moreover, set 
\[
j_C:=\max\left\{1\leq i\leq k\,|\,p_i\in C^{X_{i-1}}\right\}.
\]
By Lemma \ref{mono_lem}, we have $j_C\in\{1$, $k\}$. 
From Proposition \ref{KpM_prop}, 
we have 
\[
A_{X, \Delta}(F)=\begin{cases}
b\beta+a & \text{if }j_C=1, \\
a\beta+b & \text{if }j_C=k.
\end{cases}\]

Assume that $j_C=1$ and $k\geq 3$. Then we have 
\[
\frac{A_{X, \Delta}(F)}{A_{X, 0}(F)}>\frac{\beta+2}{3}
\]
as in Theorem \ref{S7_thm} Step 3. Take any birational morphism 
$\theta\colon X\to X'$ over $\pr^1$ such that $(X', \theta_*C)$ is of type 
$(\operatorname{I.9B.}1)$ as in Lemma \ref{easy_lem}. 
Obviously, $\theta$ is an isomorphism at $c_X(F)$. 
By Proposition \ref{easy_prop}, we have 
\[
S_L(F)\leq\frac{3\beta+4}{(4-n)\beta+4}\cdot S_{L'}(F),
\]
where $L':=-(K_{X'}+\theta_*\Delta)$. As we have already seen in 
Theorem \ref{S7_thm} Step 1, we have 
\[
\frac{A_{X,0}(F)}{S_{L'}(F)}\geq \frac{3(3\beta+4)}{7\beta^2+12\beta+6}. 
\]
Thus we have 
\begin{eqnarray*}
\frac{A_{X,\Delta}(F)}{S_{L}(F)}&>&\frac{((4-n)\beta+4)(\beta+2)}{7\beta^2+12\beta+6}\\
&>&\frac{3((4-n)\beta+4)(2\beta+1)}{(n^2-10n+24)\beta^2
+(-6n+36)\beta+12}.
\end{eqnarray*}

\noindent\underline{\textbf{Step 2}}\\
Set $l:=\eta^{-1}(\eta(q_1))$ and 
\[
D:=\beta C+l\sim_\Q L. 
\]
Assume that $j_C=1$. We may assume that $k\leq 2$ by Step 1. 
Then we can inductively check that 
\[
\pi^*D=D^{\tilde{X}}+\sum_{i=1}^m b_i(\beta+1)F_i^{\tilde{X}}. 
\]
By Proposition \ref{KpM_prop}, 
$\pi^*D-(x/(ab))F_m^*=\nu^*(\sigma^*D-x F)$ is effective for 
$x\in[0$, $b(\beta+1)]$. We have 
\begin{eqnarray*}
\left(\pi^*L-\frac{x}{ab}F_m^*\right)\cdot C^{\tilde{X}}&=&
(4-n)\beta+2-\frac{x}{a}, \\
\left(\pi^*L-\frac{x}{ab}F_m^*\right)\cdot l^{\tilde{X}}&=&
2\beta-\frac{x}{a}, \\
\left(\pi^*L-\frac{x}{ab}F_m^*\right)\cdot F^{\tilde{X}}&=&
\frac{x}{ab}.
\end{eqnarray*}
Note that $a((4-n)\beta+2)>2a\beta$ and $b(\beta+1)>2a\beta$. Thus 
$\sigma^*L-x F$ is nef for $x\in[0$, $2a\beta]$, and $\sigma^*L-2a\beta F$ induces 
the birational contraction $\phi\colon Y\to Y'$ of $l^Y$. Set 
$\nu':=\phi\circ\nu\colon\tilde{X}\to Y'$. 
By Lemma \ref{mono_lem}, we have $\left(l^{Y\cdot 2}\right)=-b/a$. 
Thus we have 
\begin{equation}\label{2_1_eqn}
\pi^*D-\frac{x}{ab}F_m^*={\nu'}^*\nu'_*\left(\pi^*D-\frac{x}{ab}F_m^*\right)
+\frac{x-2a\beta}{b}\nu^*l^Y.
\end{equation}
For any $x\in[2a\beta$, $b(\beta+1)]$, we have 
$\nu'_*\left(\pi^*D-(x/(ab))F_m^*\right)\geq 0$ and 
\begin{eqnarray*}
\left(\pi^*L-\frac{x}{ab}F_m^*-\frac{x-2a\beta}{b}\nu^*l^Y
\right)\cdot C^{\tilde{X}}&=&
\frac{2\left(2a\beta+b\left(\left(1-\frac{n}{2}\right)\beta+1\right)-x\right)}{b}, \\
\left(\pi^*L-\frac{x}{ab}F_m^*-\frac{x-2a\beta}{b}\nu^*l^Y
\right)\cdot F^{\tilde{X}}&=&
\frac{2\beta}{b}.
\end{eqnarray*}
Note that $b((3-n)\beta/2+1)\leq
\min\{b(\beta+1)$, $2a\beta+b((2-n)\beta/2+1)\}$. Thus, 
for any $x\in[2a\beta$, $b((3-n)\beta/2+1)]$, \eqref{2_1_eqn} gives the Zariski 
decomposition, hence
\begin{eqnarray*}
\vol_Y(\sigma^*L-x F)&=&\left(\pi^*L-\frac{x}{ab}F_m^*-
\frac{x-2a\beta}{b}\nu^*l^Y\right)^{\cdot 2}\\
&=&(4-n)\beta^2+4\beta-\frac{x^2}{ab}+\frac{(x-2a\beta)^2}{ab}.
\end{eqnarray*}
By Proposition \ref{pltK_prop}, for $x\geq b((3-n)\beta/2+1)$, 
\begin{eqnarray*}
&&\vol_Y(\sigma^*L-x F)\\
&\leq& \frac{4}{b(4a+(n-2)b)}\left(x-\left(b\left(\frac{(3-n)\beta}{2}+1\right)
+\frac{(4a+(n-2)b)\beta}{2}\right)\right)^2.
\end{eqnarray*}
In particular, we have 
\begin{eqnarray*}
&&S(F)\\
&\leq&\frac{1}{(4-n)\beta^2+4\beta}\Biggl(\int_0^{b((3-n)\beta/2+1)}
\left((4-n)\beta^2
+4\beta-\frac{x^2}{ab}\right)dx\\
&+&\int_{2a\beta}^{b((3-n)\beta/2+1)}
\frac{(x-2a\beta)^2}{ab}dx\\
&+&\int_{b((3-n)\beta/2+1)}^{b((3-n)\beta/2+1)+\frac{(4a+(n-2)b)\beta}{2}}
\frac{4\biggl(x-\biggl(b\left(\frac{(3-n)\beta}{2}+1\right)
+\frac{(4a+(n-2)b)\beta}{2}\biggr)\biggr)^2}{b(4a+(n-2)b)}dx\Biggr)\\
&=&\frac{a((-4n+20)\beta^2+24\beta)+b((n^2-7n+13)\beta^2+(-6n+24)
\beta+12)}{6((4-n)\beta+4)}. 
\end{eqnarray*}
Thus we have 
\begin{eqnarray*}
\frac{A_{X, \Delta}(F)}{S(F)}&\geq&
\frac{6((4-n)\beta+4)(\beta+(a/b))}{(a/b)((-4n+20)\beta^2+24\beta)
+(n^2-7n+13)\beta^2+(-6n+24)\beta+12}\\
&\geq&
\frac{6((4-n)\beta+4)(\beta+1)}{(n^2-11n+33)\beta^2+(-6n+48)\beta+12}>\frac{3}{2}
\end{eqnarray*}
since $a/b\in[1$, $2]$. 

\noindent\underline{\textbf{Step 3}}\\
Assume that $j_C=k\geq 3$. 
For $D=\beta C+l$, we can inductively check that 
\[
\pi^*D=D^{\tilde{X}}+(\beta+1)F_1^{\tilde{X}}+\sum_{i=2}^m (a_i\beta+2b_i)F_i^{\tilde{X}}. 
\]
By Proposition \ref{KpM_prop}, 
$\pi^*D-(x/(ab))F_m^*=\nu^*(\sigma^*D-x F)$ is effective for 
$x\in[0$, $a\beta+2b]$. We have 
\begin{eqnarray*}
\left(\pi^*L-\frac{x}{ab}F_m^*\right)\cdot C^{\tilde{X}}&=&
(4-n)\beta+2-\frac{x}{b}, \\
\left(\pi^*L-\frac{x}{ab}F_m^*\right)\cdot l^{\tilde{X}}&=&
\frac{2(a\beta-x)}{a}, \\
\left(\pi^*L-\frac{x}{ab}F_m^*\right)\cdot F^{\tilde{X}}&=&
\frac{x}{ab}.
\end{eqnarray*}

\underline{The case $a/b\geq ((4-n)\beta+2)/\beta$}\,\,
We have $b((4-n)\beta+2)\leq a\beta$ in this case. 
Thus 
$\sigma^*L-x F$ is nef for $x\in[0$, $b((4-n)\beta+2)]$. By Proposition \ref{pltK_prop}, 
\[
S(F)\leq\frac{a((4-n)\beta^2+4\beta)+b((4-n)\beta+2)^2}{3((4-n)\beta+2)}. 
\]
Thus we have 
\begin{eqnarray*}
\frac{A_{X, \Delta}(F)}{S(F)}&\geq&
\frac{3((4-n)\beta+2)((a/b)\beta+1)}{(a/b)((4-n)\beta^2+4\beta)
+((4-n)\beta+2)^2}.
\end{eqnarray*}
Note that $a/b\in[((4-n)\beta+2)/\beta$, $\infty)$. 
If $n\leq 4$, then 
\[
\frac{A_{X, \Delta}(F)}{S(F)}\geq\frac{3}{2}; 
\]
if $n\geq 5$, then 
\begin{eqnarray*}
\frac{A_{X, \Delta}(F)}{S(F)}>\frac{3((4-n)\beta+2)}{(4-n)\beta+4}
>\frac{3((4-n)\beta+4)(2\beta+1)}{(n^2-10n+24)\beta^2
+(-6n+36)\beta+12}.
\end{eqnarray*}

\underline{The case $a/b<((4-n)\beta+2)/\beta$}\,\,
We have $a\beta<b((4-n)\beta+2)$ in this case. 
Thus 
$\sigma^*L-x F$ is nef for $x\in[0$, $a\beta]$, 
and $\sigma^*L-a\beta F$ induces 
the birational contraction $\phi\colon Y\to Y'$ of $l^Y$. Set 
$\nu':=\phi\circ\nu\colon\tilde{X}\to Y'$. 
By Lemma \ref{mono_lem}, we have $\left(l^{Y\cdot 2}\right)=-4b/a$. 
Thus we have 
\begin{equation}\label{3_1_eqn}
\pi^*D-\frac{x}{ab}F_m^*={\nu'}^*\nu'_*\left(\pi^*D-\frac{x}{ab}F_m^*\right)
+\frac{x-a\beta}{2b}\nu^*l^Y.
\end{equation}
For any $x\in[a\beta$, $a\beta+2b]$, we have 
$\nu'_*\left(\pi^*D-(x/(ab))F_m^*\right)\geq 0$ and 
\begin{eqnarray*}
\left(\pi^*L-\frac{x}{ab}F_m^*-\frac{x-a\beta}{2b}\nu^*l^Y
\right)\cdot C^{\tilde{X}}&=&
\frac{b((4-n)\beta+2)-x}{b}, \\
\left(\pi^*L-\frac{x}{ab}F_m^*-\frac{x-a\beta}{2b}\nu^*l^Y
\right)\cdot F^{\tilde{X}}&=&
\frac{\beta}{b}.
\end{eqnarray*}
Note that the condition $a\beta+2b<b((4-n)\beta+2)$ is equivalent to the condition 
$a/b<4-n$. Since $a/b>2$, this occurs only when $n=1$. 
For any $x\in[a\beta$, $\min\{a\beta+2b$, $b((4-n)\beta+2)\}]$, 
\eqref{3_1_eqn} gives the Zariski decomposition, hence
\begin{eqnarray*}
\vol_Y(\sigma^*L-x F)&=&\left(\pi^*L-\frac{x}{ab}F_m^*-
\frac{x-a\beta}{2b}\nu^*l^Y\right)^{\cdot 2}\\
&=&(4-n)\beta^2+4\beta-\frac{x^2}{ab}+\frac{(x-a\beta)^2}{ab}.
\end{eqnarray*}

Assume that $a/b\geq 4-n$. 
By Proposition \ref{pltK_prop}, for $x\geq b((4-n)\beta+2)$, 
\[
\vol_Y(\sigma^*L-x F)\leq \frac{1}{b(a+(n-4)b)}\left(x-\left(b((4-n)\beta+2)
+(a+(n-4)b)\beta\right)\right)^2.
\]
In particular, we have 
\begin{eqnarray*}
&&S(F)\\
&\leq&\frac{1}{(4-n)\beta^2+4\beta}\biggl(\int_0^{b((4-n)\beta+2)}
\left((4-n)\beta^2+4\beta-\frac{x^2}{ab}\right)dx\\
&+&\int_{a\beta}^{b((4-n)\beta+2)}\frac{(x-a\beta)^2}{ab}dx\\
&+&\int_{b((4-n)\beta+2)}^{b((4-n)\beta+2)+(a+(n-4)b)\beta}
\frac{\left(x-\left(b((4-n)\beta+2)+(a+(n-4)b)\beta\right)\right)^2}{b(a+(n-4)b)}
dx\biggr)\\
&=&\frac{a((-n+4)\beta^2+6\beta)+b((n^2-8n+16)\beta^2+(-6n+24)\beta
+12)}{3((4-n)\beta+4)}. 
\end{eqnarray*}
Thus we have 
\begin{eqnarray*}
\frac{A_{X, \Delta}(F)}{S(F)}&\geq&
\frac{3((4-n)\beta+4)((a/b)\beta+1)}{(a/b)((-n+4)\beta^2+6\beta)+
(n^2-8n+16)\beta^2+(-6n+24)\beta+12}\\
&>&\frac{3((4-n)\beta+4)(2\beta+1)}{(n^2-10n+24)\beta^2+(-6n+36)\beta+12}
\end{eqnarray*}
since $a/b\in(2$, $((4-n)\beta+2)/\beta)$. 

Assume that $n=1$ and $a/b< 3$. 
By Proposition \ref{pltK_prop}, for $x\geq a\beta+2b$, 
\[
\vol_Y(\sigma^*L-x F)\leq \frac{1}{b(3b-a)}\left(x-\left(a\beta+2b
+(3b-a)\beta\right)\right)^2.
\]
In particular, we have 
\begin{eqnarray*}
S(F)&\leq&\frac{1}{3\beta^2+4\beta}\biggl(\int_0^{a\beta+2b}
\left(3\beta^2+4\beta-\frac{x^2}{ab}\right)dx
+\int_{a\beta}^{a\beta+2b}\frac{(x-a\beta)^2}{ab}dx\\
&&+\int_{a\beta+2b}^{a\beta+2b+(3b-a)\beta}
\frac{\left(x-\left(a\beta+2b+(3b-a)\beta\right)\right)^2}{b(3b-a)}
dx\biggr)\\
&=&\frac{a(\beta^2+2\beta)+b(3\beta^2+6\beta+4)}{3\beta+4}. 
\end{eqnarray*}
Thus we have 
\begin{eqnarray*}
\frac{A_{X, \Delta}(F)}{S(F)}&\geq&
\frac{(3\beta+4)((a/b)\beta+1)}{(a/b)(\beta^2+2\beta)+
3\beta^2+6\beta+4}\\
&>&\frac{(3\beta+4)(2\beta+1)}{5\beta^2+10\beta+4}
\end{eqnarray*}
since $a/b\in(2$, $3)$. 

\noindent\underline{\textbf{Step 4}}\\
Assume that $j_C=k=2$. 
For $D=\beta C+l$, we can inductively check that 
\[
\pi^*D=D^{\tilde{X}}+\sum_{i=1}^m a_i(\beta+1)F_i^{\tilde{X}}. 
\]
By Proposition \ref{KpM_prop}, 
$\pi^*D-(x/(ab))F_m^*=\nu^*(\sigma^*D-x F)$ is effective for 
$x\in[0$, $a(\beta+1)]$. We have 
\begin{eqnarray*}
\left(\pi^*L-\frac{x}{ab}F_m^*\right)\cdot C^{\tilde{X}}&=&
(4-n)\beta+2-\frac{x}{b}, \\
\left(\pi^*L-\frac{x}{ab}F_m^*\right)\cdot l^{\tilde{X}}&=&
\frac{2b\beta-x}{b}, \\
\left(\pi^*L-\frac{x}{ab}F_m^*\right)\cdot F^{\tilde{X}}&=&
\frac{x}{ab}.
\end{eqnarray*}
Note that $2b\beta<a(\beta+1)$ and $2b\beta<b((4-n)\beta+2)$. 
Thus 
$\sigma^*L-x F$ is nef for $x\in[0$, $2b\beta]$, 
and $\sigma^*L-2b\beta F$ induces 
the birational contraction $\phi\colon Y\to Y'$ of $l^Y$. Set 
$\nu':=\phi\circ\nu\colon\tilde{X}\to Y'$. 
By Lemma \ref{mono_lem}, we have $\left(l^{Y\cdot 2}\right)=-a/b$. 
Thus we have 
\begin{equation}\label{4_1_eqn}
\pi^*D-\frac{x}{ab}F_m^*={\nu'}^*\nu'_*\left(\pi^*D-\frac{x}{ab}F_m^*\right)
+\frac{x-2b\beta}{a}\nu^*l^Y.
\end{equation}
For any $x\in[2b\beta$, $a(\beta+1)]$, we have 
$\nu'_*\left(\pi^*D-(x/(ab))F_m^*\right)\geq 0$ and 
\begin{eqnarray*}
\left(\pi^*L-\frac{x}{ab}F_m^*-\frac{x-2b\beta}{a}\nu^*l^Y
\right)\cdot C^{\tilde{X}}&=&
\frac{2\left(a\left(-\frac{n}{2}+1\right)\beta+2b\beta+a-x\right)}{a}, \\
\left(\pi^*L-\frac{x}{ab}F_m^*-\frac{x-2b\beta}{a}\nu^*l^Y
\right)\cdot F^{\tilde{X}}&=&
\frac{2\beta}{a}.
\end{eqnarray*}
Note that the condition $a(\beta+1)<a(-n+2)\beta/2+2b\beta+a$ is equivalent to the 
condition 
$a/b<4/n$. 
For any $x\in[2b\beta$, $\min\{a(\beta+1)$, $a(-n+2)\beta/2+2b\beta+a\}]$, 
\eqref{4_1_eqn} gives the Zariski decomposition, hence
\begin{eqnarray*}
\vol_Y(\sigma^*L-x F)&=&\left(\pi^*L-\frac{x}{ab}F_m^*-
\frac{x-2b\beta}{a}\nu^*l^Y\right)^{\cdot 2}\\
&=&(4-n)\beta^2+4\beta-\frac{x^2}{ab}+\frac{(x-2b\beta)^2}{ab}.
\end{eqnarray*}

\underline{The case $a/b\geq 4/n$}\,\,
This occurs only when $n\geq 2$ since $a/b\leq 2$. 
We have $a(-n+2)\beta/2+2b\beta+a\leq a(\beta+1)$ in this case. 
By Proposition \ref{pltK_prop}, for $x\geq a(-n+2)\beta/2+2b\beta+a$, 
\[
\vol_Y(\sigma^*L-x F)\leq \frac{4}{a(na-4b)}\left(x-\left(a\frac{-n+2}{2}\beta+2b\beta+a
+\frac{na-4b}{2}\beta\right)\right)^2.
\]
In particular, we have 
\begin{eqnarray*}
&&S(F)\\
&\leq&\frac{1}{(4-n)\beta^2+4\beta}\biggl(\int_0^{a\frac{-n+2}{2}\beta+2b\beta+a}
\left((4-n)\beta^2+4\beta-\frac{x^2}{ab}\right)dx\\
&+&\int_{2b\beta}^{a\frac{-n+2}{2}\beta+2b\beta+a}\frac{(x-2b\beta)^2}{ab}dx\\
&+&\int_{a\frac{-n+2}{2}\beta+2b\beta+a}^{a\frac{-n+2}{2}\beta+2b\beta+a
+\frac{na-4b}{2}\beta}
\frac{4\left(x-\left(a\frac{-n+2}{2}\beta+2b\beta+a+\frac{na-4b}{2}\beta\right)\right)^2}{a(na-4b)}
dx\biggr)\\
&=&\frac{a((n^2-6n+12)\beta^2+(-6n+24)\beta+12)+b((-8n+24)\beta^2
+24\beta)}{6((4-n)\beta+4)}. 
\end{eqnarray*}
Thus we have 
\begin{eqnarray*}
&&\frac{A_{X, \Delta}(F)}{S(F)}\\
&\geq&
\frac{6((4-n)\beta+4)((a/b)\beta+1)}{(a/b)((n^2-6n+12)\beta^2+(-6n+24)\beta+12)+
(-8n+24)\beta^2+24\beta}\\
&\geq&\frac{3((4-n)\beta+4)(2\beta+1)}{(n^2-10n+24)\beta^2+(-6n+36)\beta+12}
\end{eqnarray*}
since $a/b\in(1$, $2]$. Moreover, equality holds only if $a/b=2$. 

Conversely, assume that $(a$, $b)=(2$, $1)$, $j_C=k=2$ and $n\geq 2$. 
Obviously, $F$ is plt-type over $(X, \Delta)$. Moreover, since $-K_{\tilde{X}}$ is big, 
$F$ is dreamy over $(X, \Delta)$ by \cite{TVAV}. If $n=2$, then 
\[
\vol_Y(\sigma^*L-(2\beta+2)F)=0.
\]
Thus we have 
\[
\frac{A_{X, \Delta}(F)}{S(F)}=\frac{3(\beta+2)(2\beta+1)}{2(2\beta^2+6\beta+3)} 
\]
for the $F$. 

If $n\geq 3$, then 
$\phi_*\left(\sigma^*L-((-n+4)\beta+2)F\right)$ induces the birational 
contraction $\phi'\colon Y'\to Y''$ of $C^{Y'}$. Set 
$\nu'':=\phi'\circ\nu'\colon\tilde{X}\to Y''$. 
By Lemma \ref{mono_lem}, we have 
$\left(l^{Y\cdot 2}\right)=-2$, $\left(l^Y\cdot C^Y\right)=0$ 
and $\left(C^{Y\cdot 2}\right)=2-n$. 
Thus we have 
\begin{equation}\label{4_2_eqn}
\pi^*D-\frac{x}{2}F_m^*={\nu''}^*\nu''_*\left(\pi^*D-\frac{x}{2}F_m^*\right)
+\nu^*\left(\frac{x-2\beta}{2}l^Y+\frac{x-((4-n)\beta+2)}{n-2}C^Y\right).
\end{equation}
For any $x\in[(-n+4)\beta+2$, $2\beta+2]$, we have 
$\nu''_*\left(\pi^*D-(x/2)F_m^*\right)\geq 0$ and 
\begin{eqnarray*}
\left(\pi^*L-\frac{x}{2}F_m^*-\nu^*\left(\frac{x-2\beta}{2}l^Y
+\frac{x-((4-n)\beta+2)}{n-2}C^Y\right)\right)
\cdot F^{\tilde{X}}
=
\frac{2\beta+2-x}{n-2}.
\end{eqnarray*}
Thus, 
for any $x\in[(-n+4)\beta+2$, $2\beta+2]$, 
\eqref{4_2_eqn} gives the Zariski decomposition, hence
\begin{eqnarray*}
&&\vol_Y(\sigma^*L-x F)\\
&=&
\left(\pi^*L-\frac{x}{2}F_m^*-\nu^*\left(\frac{x-2\beta}{2}l^Y
+\frac{x-((4-n)\beta+2)}{n-2}C^Y\right)\right)^{\cdot 2}\\
&=&\frac{(x-2\beta-2)^2}{n-2}.
\end{eqnarray*}
This immediately implies that 
\[
\frac{A_{X, \Delta}(F)}{S(F)}=
\frac{3((4-n)\beta+4)(2\beta+1)}{(n^2-10n+24)\beta^2+(-6n+36)\beta+12}
\]
for the $F$. 

\underline{The case $a/b< 4/n$}\,\,
We have $a(\beta+1)<a(-n+2)\beta/2+2b\beta+a$ in this case. 
By Proposition \ref{pltK_prop}, for $x\geq a(\beta+1)$, 
\[
\vol_Y(\sigma^*L-x F)\leq \frac{4}{a(4b-na)}\left(x-\left(a(\beta+1)
+\frac{4b-na}{2}\beta\right)\right)^2.
\]
In particular, we have 
\begin{eqnarray*}
&&S(F)\\
&\leq&\frac{1}{(4-n)\beta^2+4\beta}\biggl(\int_0^{a(\beta+1)}
\left((4-n)\beta^2+4\beta-\frac{x^2}{ab}\right)dx\\
&+&\int_{2b\beta}^{a(\beta+1)}\frac{(x-2b\beta)^2}{ab}dx\\
&+&\int_{a(\beta+1)}^{a(\beta+1)
+\frac{4b-na}{2}\beta}
\frac{4\left(x-\left(a(\beta+1)+\frac{4b-na}{2}\beta\right)\right)^2}{a(4b-na)}
dx\biggr)\\
&=&\frac{a((n^2-6n+12)\beta^2+(-6n+24)\beta+12)+b((-8n+24)\beta^2
+24\beta)}{6((4-n)\beta+4)}. 
\end{eqnarray*}
Thus we have 
\begin{eqnarray*}
\frac{A_{X, \Delta}(F)}{S(F)}\geq\frac{3((4-n)\beta+4)(2\beta+1)}{(n^2-10n+24)\beta^2+(-6n+36)\beta+12}
\end{eqnarray*}
as in th case $a/b\geq 4/n$. Moreover, equality holds only if $a/b=2$ and $n=1$. 

Conversely, assume that $(a$, $b)=(2$, $1)$, $j_C=k=2$ and $n=1$. 
In this case, in the notation of Theorem \ref{S7_thm} Step 1, we can assume that 
$F$ is the toric valuation corresponds to the primitive lattice point $(-2$, $-1)\in N$. 
Obviously, $F$ is plt-type and dreamy over $(X, \Delta)$. By \cite[Corollary 7.7]{BJ}, 
\[
S(F)=\frac{2\cdot(4\beta^2+9\beta+6)+1\cdot(7\beta^2+12\beta)}{3(3\beta+4)}
=\frac{5\beta^2+10\beta+4}{3\beta+4}
\]
and hence 
\[
\frac{A_{X, \Delta}(F)}{S(F)}=\frac{(3\beta+4)(2\beta+1)}{5\beta^2+10\beta+4} 
\]
for the $F$. 
Thus we have completed the proof. 
\end{proof}

\section{Prime divisors centered at special points, II}\label{caseII_section}

In this section, we prove the following: 

\begin{thm}\label{caseII_thm}
Let $(X, C)$ be an asymptotically log del Pezzo surface of type 
$(\operatorname{I.9B.}n)$, let $\eta\colon X\to\pr^1$ be the anti-log-canonical 
morphism, and let $q_1$, $q_2\in C$ be the ramification points of 
$\eta|_C\colon C\to\pr^1$. Take any exceptional prime divisor $F$ over $X$ such that 
$c_X(F)=q_1$. Assume that $\eta^{-1}(\eta(q_1))$ is singular. Then, for any 
$\beta\in(0$, $1/(7n))\cap\Q$, we have 
\[
\frac{A_{X, \Delta}(F)}{S_L(F)}\geq 
\frac{3((4-n)\beta+4)(\beta+1)}{(n^2-9n+20)\beta^2
+(-6n+30)\beta+12},
\]
where $\Delta:=(1-\beta)C$ and $L:=-(K_X+\Delta)$. Moreover, the above 
inequality is optimal; there uniquely exists a prime divisor over $X$ satisfying the above 
conditions such that equality holds. 
\end{thm}

\begin{proof}
The following proof is divided into 4 numbers of steps. 

\noindent\underline{\textbf{Step 1}}\\
By Proposition \ref{pltK_prop}, we may assume that $F$ is dreamy and plt-type over 
$(X, \Delta)$. We follow the notations in \S \ref{plt_section}. Moreover, set 
\[
j_C:=\max\left\{1\leq i\leq k\,|\,p_i\in C^{X_{i-1}}\right\}.
\]
By Lemma \ref{mono_lem}, we have $j_C\in\{1$, $k\}$. 
From Proposition \ref{KpM_prop}, 
we have 
\[
A_{X, \Delta}(F)=\begin{cases}
b\beta+a & \text{if }j_C=1, \\
a\beta+b & \text{if }j_C=k.
\end{cases}\]

Assume that $j_C=1$ and $k\geq 3$. 
The completely same argument in Theorem \ref{caseI_thm} Step 1 shows that 
\begin{eqnarray*}
&&\frac{A_{X,\Delta}(F)}{S_{L}(F)}>\frac{((4-n)\beta+4)(\beta+2)}{7\beta^2+12\beta+6}\\
&>&\frac{3((4-n)\beta+4)(\beta+1)}{(n^2-9n+20)\beta^2
+(-6n+30)\beta+12}.
\end{eqnarray*}

\noindent\underline{\textbf{Step 2}}\\
Write $\eta^{-1}(\eta(q_1))=l_1+l_2$. We have $q_1=l_1\cap l_2$. Set  
\[
D:=\beta C+l_1+l_2\sim_\Q L. 
\]
Assume that $j_C=1$. Moreover, assume either 
\begin{itemize}
\item
$k=1$ (i.e., $m=1$), or 
\item
$k=2$ and $p_2\in F_1\setminus\left(l_1^{X_1}\cup l_2^{X_1}\right)$.
\end{itemize}
We can inductively check that 
\[
\pi^*D=D^{\tilde{X}}+\sum_{i=1}^m b_i(\beta+2)F_i^{\tilde{X}}. 
\]
By Proposition \ref{KpM_prop}, 
$\pi^*D-(x/(ab))F_m^*=\nu^*(\sigma^*D-x F)$ is effective for 
$x\in[0$, $b(\beta+2)]$. We have 
\begin{eqnarray*}
\left(\pi^*L-\frac{x}{ab}F_m^*\right)\cdot C^{\tilde{X}}&=&
(4-n)\beta+2-\frac{x}{a}, \\
\left(\pi^*L-\frac{x}{ab}F_m^*\right)\cdot l_i^{\tilde{X}}&=&
\beta-\frac{x}{a}\quad(i=1, 2), \\
\left(\pi^*L-\frac{x}{ab}F_m^*\right)\cdot F^{\tilde{X}}&=&
\frac{x}{ab}.
\end{eqnarray*}
Note that $a((4-n)\beta+2)>a\beta$ and $b(\beta+2)>a\beta$. Thus 
$\sigma^*L-x F$ is nef for $x\in[0$, $a\beta]$, and $\sigma^*L-a\beta F$ induces 
the birational contraction $\phi\colon Y\to Y'$ of $l_1^Y$ and $l_2^Y$. Set 
$\nu':=\phi\circ\nu\colon\tilde{X}\to Y'$. 
By Lemma \ref{mono_lem}, we have $\left(l_1^{Y\cdot 2}\right)=-(a+b)/a$, 
$\left(l_1^Y\cdot l_2^Y\right)=(a-b)/a$ and $\left(l_2^{Y\cdot 2}\right)=-(a+b)/a$. 
Thus we have 
\begin{equation}\label{II-1_eqn}
\pi^*D-\frac{x}{ab}F_m^*={\nu'}^*\nu'_*\left(\pi^*D-\frac{x}{ab}F_m^*\right)
+\frac{x-a\beta}{2b}\nu^*\left(l_1^Y+l_2^Y\right).
\end{equation}
For any $x\in[a\beta$, $b(\beta+2)]$, we have 
$\nu'_*\left(\pi^*D-(x/(ab))F_m^*\right)\geq 0$ and 
\begin{eqnarray*}
\left(\pi^*L-\frac{x}{ab}F_m^*-\frac{x-a\beta}{2b}\nu^*\left(l_1^Y+l_2^Y\right)
\right)\cdot C^{\tilde{X}}&=&
\frac{a\beta+(-n+3)b\beta+2b-x}{b}, \\
\left(\pi^*L-\frac{x}{ab}F_m^*-\frac{x-a\beta}{2b}\nu^*\left(l_1^Y+l_2^Y\right)
\right)\cdot F^{\tilde{X}}&=&
\frac{\beta}{b}.
\end{eqnarray*}
Note that the condition $a\beta+(-n+3)b\beta+2b\leq b(\beta+2)$ is equivalent 
to the condition $a/b\leq n-2$. This occurs only when $n\geq 3$ since $a/b\geq 1$. 
For any $x\in[a\beta$, $\min\{b(\beta+2)$, $a\beta+(-n+3)b\beta+2b\}]$, 
\eqref{II-1_eqn} gives the Zariski decomposition, hence
\begin{eqnarray*}
\vol_Y(\sigma^*L-x F)&=&
\left(\pi^*L-\frac{x}{ab}F_m^*-\frac{x-a\beta}{2b}\nu^*\left(l_1^Y+l_2^Y\right)
\right)^{\cdot 2}\\
&=&(4-n)\beta^2+4\beta-\frac{x^2}{ab}+\frac{(x-a\beta)^2}{ab}.
\end{eqnarray*}

\underline{The case $a/b\leq n-2$}\,\,
By Proposition \ref{pltK_prop}, for $x\geq a\beta+(-n+3)b\beta+2b$, 
\begin{eqnarray*}
&&\vol_Y(\sigma^*L-x F)\\
&\leq& \frac{1}{b(-a+(n-2)b)}\left(x-\left(a\beta+(-n+3)b\beta+2b
+(-a+(n-2)b)\beta\right)\right)^2.
\end{eqnarray*}
In particular, we have 
\begin{eqnarray*}
&&S(F)\\
&\leq&\frac{1}{(4-n)\beta^2+4\beta}\biggl(\int_0^{a\beta+(-n+3)b\beta+2b}
\left((4-n)\beta^2
+4\beta-\frac{x^2}{ab}\right)dx\\
&+&\int_{a\beta}^{a\beta+(-n+3)b\beta+2b}
\frac{(x-a\beta)^2}{ab}dx\\
&+&\int_{a\beta+(-n+3)b\beta+2b}^{a\beta+(-n+3)b\beta+2b+(-a+(n-2)b)\beta}
\frac{\left(x-\left(a\beta+(-n+3)b\beta+2b
+(-a+(n-2)b)\beta\right)\right)^2}{b(-a+(n-2)b)}dx\biggr)\\
&=&\frac{a((-2n+7)\beta^2+6\beta)+b((n^2-7n+13)\beta^2+(-6n+24)
\beta+12)}{3((4-n)\beta+4)}. 
\end{eqnarray*}
Thus we have 
\begin{eqnarray*}
\frac{A_{X, \Delta}(F)}{S(F)}&\geq&
\frac{3((4-n)\beta+4)(\beta+(a/b))}{(a/b)((-2n+7)\beta^2+6\beta)
+(n^2-7n+13)\beta^2+(-6n+24)\beta+12}\\
&\geq&
\frac{3((4-n)\beta+4)(\beta+1)}{(n^2-9n+20)\beta^2+(-6n+30)\beta+12}
\end{eqnarray*}
since $a/b\in[1$, $2]$. 
Moreover, equality holds only if $a/b=1$ and $n\geq 3$. 

Conversely, assume that $(a$, $b)=(1$, $1)$ and $n\geq 3$. 
Obviously, $F$ is plt-type and dreamy over $(X, \Delta)$. If $n=3$, then 
\[
\vol_Y(\sigma^*L-(\beta+2)F)=0.
\]
Thus we have 
\[
\frac{A_{X, \Delta}(F)}{S(F)}=\frac{3(\beta+4)(\beta+1)}{2(\beta^2+6\beta+6)} 
\]
for the $F$. 

If $n\geq 4$, then 
$\phi_*\left(\sigma^*L-((-n+4)\beta+2)F\right)$ induces the birational 
contraction $\phi'\colon Y'\to Y''$ of $C^{Y'}$. Set 
$\nu'':=\phi'\circ\nu'\colon\tilde{X}\to Y''$. 
By Lemma \ref{mono_lem}, we have 
$\left(l_i^{Y\cdot 2}\right)=-2$, $\left(l_i^Y\cdot C^Y\right)=0$ $(i=1$, $2)$, 
$\left(l_1^Y\cdot l_2^Y\right)=0$ and $\left(C^{Y\cdot 2}\right)=3-n$. 
Thus we have 
\begin{equation}\label{II-2_eqn}
\pi^*D-x F_m^*={\nu''}^*\nu''_*\left(\pi^*D-x F_m^*\right)
+\nu^*\left(\frac{x-((4-n)\beta+2)}{n-3}C^Y+\frac{x-\beta}{2}l_1^Y
+\frac{x-\beta}{2}l_2^Y\right).
\end{equation}
For any $x\in[(-n+4)\beta+2$, $\beta+2]$, we have 
$\nu''_*\left(\pi^*D-x F_m^*\right)\geq 0$ and 
\begin{eqnarray*}
\left(\pi^*L-x F_m^*-\nu^*\left(\frac{x-((4-n)\beta+2)}{n-3}C^Y
+\frac{x-\beta}{2}l_1^Y
+\frac{x-\beta}{2}l_2^Y\right)\right)
\cdot F^{\tilde{X}}
=
\frac{\beta+2-x}{n-3}.
\end{eqnarray*}
Thus, 
for any $x\in[(-n+4)\beta+2$, $\beta+2]$, 
\eqref{II-2_eqn} gives the Zariski decomposition, hence
\begin{eqnarray*}
&&\vol_Y(\sigma^*L-x F)\\
&=&
\left(\pi^*L-x F_m^*-\nu^*\left(\frac{x-((4-n)\beta+2)}{n-3}C^Y
+\frac{x-\beta}{2}l_1^Y
+\frac{x-\beta}{2}l_2^Y\right)\right)^{\cdot 2}\\
&=&\frac{(x-\beta-2)^2}{n-3}.
\end{eqnarray*}
This immediately implies that 
\[
\frac{A_{X, \Delta}(F)}{S(F)}=
\frac{3((4-n)\beta+4)(\beta+1)}{(n^2-9n+20)\beta^2+(-6n+30)\beta+12}
\]
for the $F$. 

\underline{The case $a/b> n-2$}\,\,
By Proposition \ref{pltK_prop}, for $x\geq b(\beta+2)$, 
\begin{eqnarray*}
&&\vol_Y(\sigma^*L-x F)\\
&\leq& \frac{1}{b(a+(-n+2)b)}\left(x-\left(
b(\beta+2)
+(a+(-n+2)b)\beta\right)\right)^2.
\end{eqnarray*}
In particular, we have 
\begin{eqnarray*}
&&S(F)\\
&\leq&\frac{1}{(4-n)\beta^2+4\beta}\biggl(\int_0^{b(\beta+2)}
\left((4-n)\beta^2+4\beta-\frac{x^2}{ab}\right)dx\\
&+&\int_{a\beta}^{b(\beta+2)}\frac{(x-a\beta)^2}{ab}dx\\
&+&\int_{b(\beta+2)}^{b(\beta+2)
+(a+(-n+2)b)\beta}
\frac{\left(x-\left(b(\beta+2)+(a+(-n+2)b)\beta\right)\right)^2}{b(a+(-n+2)b)}
dx\biggr)\\
&=&\frac{a((-2n+7)\beta^2+6\beta)+b((n^2-7n+13)\beta^2
+(-6n+24)\beta+12)}{3((4-n)\beta+4)}. 
\end{eqnarray*}
Thus we have 
\begin{eqnarray*}
\frac{A_{X, \Delta}(F)}{S(F)}
\geq
\frac{3((4-n)\beta+4)(\beta+1)}{(n^2-9n+20)\beta^2+(-6n+30)\beta+12}
\end{eqnarray*}
as in the case $a/b\leq n-2$. 
Moreover, equality holds only if $a/b=1$ and $n\in\{1$, $2\}$. 

Conversely, assume that $(a$, $b)=(1$, $1)$ and $n\in\{1$, $2\}$. 
Obviously, $F$ is plt-type and dreamy over $(X, \Delta)$. 
We consider the case $n=1$. In this case, 
in the notation of Theorem \ref{S7_thm} Step 1, we can assume that 
$F$ is the toric valuation corresponds to the primitive lattice point $(2$, $1)\in N$. 
By \cite[Corollary 7.7]{BJ}, 
\[
S(F)=-\frac{2\cdot(4\beta^2+9\beta+6)+1\cdot(7\beta^2+12\beta)}{3(3\beta+4)}
+(3\beta+2)
=\frac{4(\beta+1)^2}{3\beta+4}
\]
and hence 
\[
\frac{A_{X, \Delta}(F)}{S(F)}=\frac{3\beta+4}{4(\beta+1)} 
\]
for the $F$. 

We consider the case $n=2$. In this case, $X$ is the toric variety corresponds to 
the complete fan in $N_\R$ whose set of $1$-dimensional cones is equal to the set 
\[
\{\R_{\geq 0}(1,0),\,\,\R_{\geq 0}(1,1),\,\,\R_{\geq 0}(0,1),\,\,\R_{\geq 0}(-1,0),\,\,
\R_{\geq 0}(-1,-1),\,\,\R_{\geq 0}(0,-1)\}.
\]
Let $e_1$, $f_3$, $e_2$, $f_1$, $e_3$, $f_2$ be the torus invariant prime divisor 
corresponds to 
$\R_{\geq 0}(1,0)$, $\R_{\geq 0}(1,1)$, $\R_{\geq 0}(0,1)$, $\R_{\geq 0}(-1,0)$, 
$\R_{\geq 0}(-1,-1)$, $\R_{\geq 0}(0,-1)$, respectively. In this setting, we may assume 
that $C\sim f_2+e_1+f_3+e_2$ and $L$ is $\Q$-linearly equivalent to 
\[
D_0:=\beta f_2+(\beta+1)e_1+(\beta+1)f_3+\beta e_2, 
\]
and $F$ corresponds to the lattice point $(2$, $1)\in N$. The barycenter 
of the polytope in $M_\R$ associates with $D_0$ is equal to 
$(-(\beta+1)/2$, $0)$. By \cite[Corollary 7.7]{BJ}, we have 
$S(F)=-(\beta+1)+2(\beta+1)=\beta+1$. Thus we have 
\[
\frac{A_{X, \Delta}(F)}{S(F)}=1
\]
for the $F$. 

\noindent\underline{\textbf{Step 3}}\\
Assume that $j_C=1$, $k=2$ and $p_2\in l_1^{X_1}$. 
We can inductively check that 
\[
\pi^*D=D^{\tilde{X}}+\sum_{i=1}^m (b_i\beta+a_i+b_i)F_i^{\tilde{X}}. 
\]
By Proposition \ref{KpM_prop}, 
$\pi^*D-(x/(ab))F_m^*=\nu^*(\sigma^*D-x F)$ is effective for 
$x\in[0$, $b\beta+a+b]$. We have 
\begin{eqnarray*}
\left(\pi^*L-\frac{x}{ab}F_m^*\right)\cdot C^{\tilde{X}}&=&
(4-n)\beta+2-\frac{x}{a}, \\
\left(\pi^*L-\frac{x}{ab}F_m^*\right)\cdot l_1^{\tilde{X}}&=&
\beta-\frac{x}{b}, \\
\left(\pi^*L-\frac{x}{ab}F_m^*\right)\cdot l_2^{\tilde{X}}&=&
\beta-\frac{x}{a}, \\
\left(\pi^*L-\frac{x}{ab}F_m^*\right)\cdot F^{\tilde{X}}&=&
\frac{x}{ab}.
\end{eqnarray*}
Note that $b\beta<a((4-n)\beta+2)$. Thus 
$\sigma^*L-x F$ is nef for $x\in[0$, $b\beta]$, and $\sigma^*L-b\beta F$ induces 
the birational contraction $\phi\colon Y\to Y'$ of $l_1^Y$. Set 
$\nu':=\phi\circ\nu\colon\tilde{X}\to Y'$. 
By Lemma \ref{mono_lem}, we have $\left(l_1^{Y\cdot 2}\right)=-(a+b)/b$. 
Thus we have 
\begin{equation}\label{III-1_eqn}
\pi^*D-\frac{x}{ab}F_m^*={\nu'}^*\nu'_*\left(\pi^*D-\frac{x}{ab}F_m^*\right)
+\frac{x-b\beta}{a+b}\nu^*l_1^Y.
\end{equation}
For any $x\in[b\beta$, $b\beta+a+b]$, we have 
$\nu'_*\left(\pi^*D-(x/(ab))F_m^*\right)\geq 0$ and 
\begin{eqnarray*}
\left(\pi^*L-\frac{x}{ab}F_m^*-\frac{x-b\beta}{a+b}\nu^*l_1^Y
\right)\cdot C^{\tilde{X}}&=&
\frac{a((4-n)\beta+2)-x}{a}, \\
\left(\pi^*L-\frac{x}{ab}F_m^*-\frac{x-b\beta}{a+b}\nu^*l_1^Y
\right)\cdot l_2^{\tilde{X}}&=&
\frac{a\beta-x}{a}, \\
\left(\pi^*L-\frac{x}{ab}F_m^*-\frac{x-b\beta}{a+b}\nu^*l_1^Y
\right)\cdot F^{\tilde{X}}&=&
\frac{x+a\beta}{a(a+b)}.
\end{eqnarray*}
Note that the $a\beta<b\beta+a+b$ and $a\beta<a((4-n)\beta+2)$. 
For any $x\in[b\beta$, $a\beta]$, 
\eqref{III-1_eqn} gives the Zariski decomposition, hence
\begin{eqnarray*}
\vol_Y(\sigma^*L-x F)&=&
\left(\pi^*L-\frac{x}{ab}F_m^*-\frac{x-b\beta}{a+b}\nu^*l_1^Y
\right)^{\cdot 2}\\
&=&(4-n)\beta^2+4\beta-\frac{x^2}{ab}+\frac{(x-b\beta)^2}{b(a+b)}.
\end{eqnarray*}
Moreover, 
$\phi_*\left(\sigma^*L-a\beta F\right)$ induces the birational 
contraction $\phi'\colon Y'\to Y''$ of $l_2^{Y'}$. Set 
$\nu'':=\phi'\circ\nu'\colon\tilde{X}\to Y''$. 
By Lemma \ref{mono_lem}, we have 
$\left(l_1^{Y\cdot 2}\right)=-(a+b)/b$, $\left(l_1^Y\cdot l_2^Y\right)=0$ 
and $\left(l_2^{Y\cdot 2}\right)=-(a+b)/a$. 
Thus we have 
\begin{equation}\label{III-2_eqn}
\pi^*D-\frac{x}{ab}F_m^*={\nu''}^*\nu''_*\left(\pi^*D-\frac{x}{ab}F_m^*\right)
+\nu^*\left(\frac{x-b\beta}{a+b}l_1^Y+\frac{x-a\beta}{a+b}l_2^Y\right).
\end{equation}
For any $x\in[a\beta$, $b\beta+a+b]$, we have 
$\nu''_*\left(\pi^*D-(x/(ab))F_m^*\right)\geq 0$ and 
\begin{eqnarray*}
\left(\pi^*L-\frac{x}{ab}F_m^*-\nu^*\left(\frac{x-b\beta}{a+b}l_1^Y
+\frac{x-a\beta}{a+b}l_2^Y\right)\right)
\cdot C^{\tilde{X}}
&=&
\frac{2\left(\left(\frac{-n+5}{2}a+\frac{-n+3}{2}b\right)\beta+a+b-x\right)}{a+b}, \\
\left(\pi^*L-\frac{x}{ab}F_m^*-\nu^*\left(\frac{x-b\beta}{a+b}l_1^Y
+\frac{x-a\beta}{a+b}l_2^Y\right)\right)
\cdot F^{\tilde{X}}
&=&
\frac{2\beta}{a+b}.
\end{eqnarray*}
Note that the condition 
\[
\left(\frac{-n+5}{2}a+\frac{-n+3}{2}b\right)\beta+a+b\leq  b\beta+a+b
\]
is equivalent to the condition $(5-n)a\leq (n-1)b$. Since 
$a/b\in(1$, $2]$, this condition is equivalent to the condition $n\geq 4$. 
For any $x\in[a\beta$, $\min\{((-n+5)a/2+(-n+3)b/2)\beta+a+b$, $b\beta+a+b\}]$, 
\eqref{III-2_eqn} gives the Zariski decomposition, hence
\begin{eqnarray*}
&&\vol_Y(\sigma^*L-x F)\\
&=&
\left(\pi^*L-\frac{x}{ab}F_m^*-\nu^*\left(\frac{x-b\beta}{a+b}l_1^Y
+\frac{x-a\beta}{a+b}l_2^Y\right)\right)^{\cdot 2}\\
&=&(4-n)\beta^2+4\beta-\frac{x^2}{ab}+\frac{(x-b\beta)^2}{b(a+b)}
+\frac{(x-a\beta)^2}{a(a+b)}.
\end{eqnarray*}

\underline{The case $n\geq 4$}\,\,
By Proposition \ref{pltK_prop}, for $x\geq ((-n+5)a/2+(-n+3)b/2)\beta+a+b$, 
\[
\vol_Y(\sigma^*L-x F)\leq 
\frac{4\left(x-\left(\left(\frac{-n+5}{2}a+\frac{-n+3}{2}b\right)\beta+a+b
+\frac{(n-5)a+(n-1)b}{2}\beta\right)\right)^2}{(a+b)\left((n-5)a+(n-1)b\right)}.
\]
In particular, we have 
\begin{eqnarray*}
&&S(F)\\
&\leq&\frac{1}{(4-n)\beta^2+4\beta}\biggl(
\int_0^{\left(\frac{-n+5}{2}a+\frac{-n+3}{2}b\right)\beta+a+b}\left((4-n)\beta^2
+4\beta-\frac{x^2}{ab}\right)dx\\
&&+\int_{b\beta}^{\left(\frac{-n+5}{2}a+\frac{-n+3}{2}b\right)\beta+a+b}
\frac{(x-b\beta)^2}{b(a+b)}dx+
\int_{a\beta}^{\left(\frac{-n+5}{2}a+\frac{-n+3}{2}b\right)\beta+a+b}
\frac{(x-a\beta)^2}{a(a+b)}dx\\
&&+\int_{\left(\frac{-n+5}{2}a+\frac{-n+3}{2}b\right)\beta+a+b}^{\left(\frac{-n+5}{2}a+\frac{-n+3}{2}b\right)\beta+a+b+\frac{((n-5)a+(n-1)b)\beta}{2}}
\frac{4}{(a+b)\left((n-5)a+(n-1)b\right)}\\
&&\quad\cdot \left(x-\left(\left(\frac{-n+5}{2}a+\frac{-n+3}{2}b\right)\beta+a+b
+\frac{((n-5)a+(n-1)b)\beta}{2}\right)\right)^2dx\biggr)\\
&=&\frac{a((n^2-10n+23)\beta^2+(-6n+30)\beta+12)
+b((n^2-8n+17)\beta^2+(-6n+30)\beta+12)}{6((4-n)\beta+4)}. 
\end{eqnarray*}
Thus we have 
\begin{eqnarray*}
&&\frac{A_{X, \Delta}(F)}{S(F)}\\
&\geq&
\frac{6((4-n)\beta+4)(\beta+(a/b))}{(a/b)
((n^2-10n+23)\beta^2+(-6n+30)\beta+12)+(n^2-8n+17)\beta^2+(-6n+30)\beta+12}\\
&>&\frac{3((4-n)\beta+4)(\beta+1)}{(n^2-9n+20)\beta^2+(-6n+30)\beta+12}
\end{eqnarray*}
since $a/b\in(1$, $2]$. 

\underline{The case $n\leq 3$}\,\,
By Proposition \ref{pltK_prop}, for $x\geq b\beta+a+b$, 
\[
\vol_Y(\sigma^*L-x F)\leq 
\frac{4\left(x-\left(b\beta+a+b
+\frac{(-n+5)a+(-n+1)b}{2}\beta\right)\right)^2}{(a+b)\left((-n+5)a+(-n+1)b\right)}.
\]
In particular, we have 
\begin{eqnarray*}
&&S(F)\\
&\leq&\frac{1}{(4-n)\beta^2+4\beta}\biggl(
\int_0^{b\beta+a+b}\left((4-n)\beta^2
+4\beta-\frac{x^2}{ab}\right)dx\\
&&+\int_{b\beta}^{b\beta+a+b}
\frac{(x-b\beta)^2}{b(a+b)}dx+
\int_{a\beta}^{b\beta+a+b}
\frac{(x-a\beta)^2}{a(a+b)}dx\\
&&+\int_{b\beta+a+b}^{b\beta+a+b+\frac{(-n+5)a+(-n+1)b}{2}\beta}
\frac{4\left(x-\left(b\beta+a+b
+\frac{(-n+5)a+(-n+1)b}{2}\beta\right)\right)^2}{(a+b)\left((-n+5)a+(-n+1)b\right)}
dx\biggr)\\
&=&\frac{a((n^2-10n+23)\beta^2+(-6n+30)\beta+12)
+b((n^2-8n+17)\beta^2+(-6n+30)\beta+12)}{6((4-n)\beta+4)}. 
\end{eqnarray*}
Thus, as in the case $n\geq 4$, we have 
\begin{eqnarray*}
\frac{A_{X, \Delta}(F)}{S(F)}>
\frac{3((4-n)\beta+4)(\beta+1)}{(n^2-9n+20)\beta^2+(-6n+30)\beta+12}.
\end{eqnarray*}

\noindent\underline{\textbf{Step 4}}\\
Assume that $j_C=k\geq 2$. 
We can inductively check that 
\[
\pi^*D=D^{\tilde{X}}+\sum_{i=1}^m (a_i\beta+2b_i)F_i^{\tilde{X}}. 
\]
By Proposition \ref{KpM_prop}, 
$\pi^*D-(x/(ab))F_m^*=\nu^*(\sigma^*D-x F)$ is effective for 
$x\in[0$, $a\beta+2b]$. We have 
\begin{eqnarray*}
\left(\pi^*L-\frac{x}{ab}F_m^*\right)\cdot C^{\tilde{X}}&=&
(4-n)\beta+2-\frac{x}{b}, \\
\left(\pi^*L-\frac{x}{ab}F_m^*\right)\cdot l_i^{\tilde{X}}&=&
\frac{a\beta-x}{a} \quad (i=1,2), \\
\left(\pi^*L-\frac{x}{ab}F_m^*\right)\cdot F^{\tilde{X}}&=&
\frac{x}{ab}.
\end{eqnarray*}

\underline{The case $a/b\geq ((4-n)\beta+2)/\beta$}\,\,
We have $b((4-n)\beta+2)\leq a\beta$ in this case. 
Thus 
$\sigma^*L-x F$ is nef for $x\in[0$, $b((4-n)\beta+2)]$. By Proposition \ref{pltK_prop}, 
\[
S(F)\leq\frac{a((4-n)\beta^2+4\beta)+b((4-n)\beta+2)^2}{3((4-n)\beta+2)}. 
\]
Thus we have 
\begin{eqnarray*}
\frac{A_{X, \Delta}(F)}{S(F)}&\geq&
\frac{3((4-n)\beta+2)((a/b)\beta+1)}{(a/b)((4-n)\beta^2+4\beta)
+((4-n)\beta+2)^2}.
\end{eqnarray*}
Note that $a/b\in[((4-n)\beta+2)/\beta$, $\infty)$. 
If $n\leq 4$, then 
\[
\frac{A_{X, \Delta}(F)}{S(F)}\geq\frac{3}{2}; 
\]
if $n\geq 5$, then 
\begin{eqnarray*}
&&\frac{A_{X, \Delta}(F)}{S(F)}>\frac{3((4-n)\beta+2)}{(4-n)\beta+4}\\
&>&\frac{3((4-n)\beta+4)(\beta+1)}{(n^2-9n+20)\beta^2
+(-6n+30)\beta+12}.
\end{eqnarray*}

\underline{The case $a/b<((4-n)\beta+2)/\beta$}\,\,
We have $a\beta<b((4-n)\beta+2)$ in this case. 
Thus 
$\sigma^*L-x F$ is nef for $x\in[0$, $a\beta]$, 
and $\sigma^*L-a\beta F$ induces 
the birational contraction $\phi\colon Y\to Y'$ of $l_1^Y$ and $l_2^Y$. Set 
$\nu':=\phi\circ\nu\colon\tilde{X}\to Y'$. 
By Lemma \ref{mono_lem}, we have $\left(l_1^{Y\cdot 2}\right)=-(a+b)/a$, 
$\left(l_1^Y\cdot l_2^Y\right)=(a-b)/a$ and $\left(l_2^{Y\cdot 2}\right)=-(a+b)/a$. 
Thus we have 
\begin{equation}\label{IV-1_eqn}
\pi^*D-\frac{x}{ab}F_m^*={\nu'}^*\nu'_*\left(\pi^*D-\frac{x}{ab}F_m^*\right)
+\frac{x-a\beta}{2b}\nu^*\left(l_1^Y+l_2^Y\right).
\end{equation}
For any $x\in[a\beta$, $a\beta+2b]$, we have 
$\nu'_*\left(\pi^*D-(x/(ab))F_m^*\right)\geq 0$ and 
\begin{eqnarray*}
\left(\pi^*L-\frac{x}{ab}F_m^*-\frac{x-a\beta}{2b}\nu^*\left(l_1^Y+l_2^Y\right)
\right)\cdot C^{\tilde{X}}&=&
\frac{b((4-n)\beta+2)-x}{b}, \\
\left(\pi^*L-\frac{x}{ab}F_m^*-\frac{x-a\beta}{2b}\nu^*\left(l_1^Y+l_2^Y\right)
\right)\cdot F^{\tilde{X}}&=&
\frac{\beta}{b}.
\end{eqnarray*}
Note that the condition $a\beta+2b<b((4-n)\beta+2)$ is equivalent to the condition 
$a/b<4-n$. 
For any $x\in[a\beta$, $\min\{a\beta+2b$, $b((4-n)\beta+2)\}]$, 
\eqref{IV-1_eqn} gives the Zariski decomposition, hence
\begin{eqnarray*}
\vol_Y(\sigma^*L-x F)&=&\left(\pi^*L-\frac{x}{ab}F_m^*
-\frac{x-a\beta}{2b}\nu^*(l_1^Y+l_2^Y)
\right)^{\cdot 2}\\
&=&(4-n)\beta^2+4\beta-\frac{x^2}{ab}+\frac{(x-a\beta)^2}{ab}.
\end{eqnarray*}

Assume that $a/b\geq 4-n$. 
By Proposition \ref{pltK_prop}, for $x\geq b((4-n)\beta+2)$, 
\[
\vol_Y(\sigma^*L-x F)\leq \frac{1}{b(a+(n-4)b)}\left(x-\left(b((4-n)\beta+2)
+(a+(n-4)b)\beta\right)\right)^2.
\]
In particular, we have 
\begin{eqnarray*}
&&S(F)\\
&\leq&\frac{1}{(4-n)\beta^2+4\beta}\biggl(\int_0^{b((4-n)\beta+2)}
\left((4-n)\beta^2+4\beta-\frac{x^2}{ab}\right)dx\\
&+&\int_{a\beta}^{b((4-n)\beta+2)}\frac{(x-a\beta)^2}{ab}dx\\
&+&\int_{b((4-n)\beta+2)}^{b((4-n)\beta+2)+(a+(n-4)b)\beta}
\frac{\left(x-\left(b((4-n)\beta+2)+(a+(n-4)b)\beta\right)\right)^2}{b(a+(n-4)b)}
dx\biggr)\\
&=&\frac{a((-n+4)\beta^2+6\beta)+b((n^2-8n+16)\beta^2+(-6n+24)\beta
+12)}{3((4-n)\beta+4)}. 
\end{eqnarray*}
Thus we have 
\begin{eqnarray*}
\frac{A_{X, \Delta}(F)}{S(F)}
&\geq&\frac{3((4-n)\beta+4)((a/b)\beta+1)}{(a/b)((-n+4)\beta^2+6\beta)
+(n^2-8n+16)\beta^2+(-6n+24)\beta+12}\\
&>&\frac{3((4-n)\beta+4)(\beta+1)}{(n^2-9n+20)\beta^2+(-6n+30)\beta+12}
\end{eqnarray*}
since $a/b\in(1$, $((4-n)\beta+2)/\beta)$. 

Assume that $a/b<4-n$. 
By Proposition \ref{pltK_prop}, for $x\geq a\beta+2b$, 
\[
\vol_Y(\sigma^*L-x F)\leq \frac{1}{b(-a+(4-n)b)}\left(x-\left(a\beta+2b
+(-a+(4-n)b)\beta\right)\right)^2.
\]
In particular, we have 
\begin{eqnarray*}
S(F)&\leq&\frac{1}{(4-n)\beta^2+4\beta}\biggl(\int_0^{a\beta+2b}
\left((4-n)\beta^2+4\beta-\frac{x^2}{ab}\right)dx
+\int_{a\beta}^{a\beta+2b}\frac{(x-a\beta)^2}{ab}dx\\
&&+\int_{a\beta+2b}^{a\beta+2b+(-a+(4-n)b)\beta}
\frac{\left(x-\left(a\beta+2b+(-a+(4-n)b)\beta\right)\right)^2}{b(-a+(4-n)b)}
dx\biggr)\\
&=&\frac{a((-n+4)\beta^2+6\beta)+b((n^2-8n+16)\beta^2+(-6n+24)\beta
+12)}{3((4-n)\beta+4)}. 
\end{eqnarray*}
Thus, as in the case $(4-n)b\leq a$, we have 
\begin{eqnarray*}
\frac{A_{X, \Delta}(F)}{S(F)}
>\frac{3((4-n)\beta+4)(\beta+1)}{(n^2-9n+20)\beta^2+(-6n+30)\beta+12}.
\end{eqnarray*}
Thus we have completed the proof. 
\end{proof}

\section{Proof of Theorem \ref{mainthm}}\label{proof_thm}

\begin{proof}[Proof of Theorem \ref{mainthm}]
Take any prime divisor $F$ over $X$. Assume that $c_X(F)\not\in\{q_1$, $q_2\}$. 
Take a birational morphism $\theta\colon X\to X'$ over $\pr^1$ as in 
Lemma \ref{easy_lem} such that $\theta$ is an isomorphism at the generic point of 
$c_X(F)$ and $(X', \theta_*C)$ is of type $(\operatorname{I.9B.}1)$. 
By Proposition \ref{easy_prop}, we have 
\[
\frac{A_{X, \Delta}(F)}{S_L(F)}\geq 
\frac{(4-n)\beta+4}{3\beta+4}\cdot\frac{A_{X', \Delta'}(F)}{S_{L'}(F)},
\]
where $\Delta':=\theta_*\Delta$ and $L':=-(K_{X'}+\Delta')$. 
By Theorem \ref{S7_thm}, we have 
\[
\frac{A_{X', \Delta'}(F)}{S_{L'}(F)}\geq \frac{6}{5}.
\]
Therefore we have 
\[
\frac{A_{X, \Delta}(F)}{S_L(F)}\geq\frac{6((4-n)\beta+4)}{5(3\beta+4)}
>\frac{3((4-n)\beta+4)(2\beta+1)}{(n^2-10n+24)\beta^2
+(-6n+36)\beta+12}.
\]
Thus we can reduced to the case $F$ is exceptional over $X$ with 
$c_X(F)\in\{q_1$, $q_2\}$. Thus we have completed the proof of Theorem \ref{mainthm} 
by Theorems \ref{caseI_thm} and \ref{caseII_thm}. 
\end{proof}

\begin{remark}\label{61_rmk}
Assume that $n=2$ and both $q_1$ and $q_2$ lie on singular fibers of $\eta$. We may 
assume that $X$ is the blowup of $\pr^1_{z_{10}:z_{11}}\times\pr^1_{z_{20}:z_{21}}$ along 
$(0:1;0:1)$ and $(1:0;1:0)$, and $C$ is the strict transform of 
\[
C':=\{z_{10}z_{21}^2=z_{11}z_{20}^2\}\subset\pr^1_{z_{10}:z_{11}}\times\pr^1_{z_{20}:z_{21}}.
\]
Let us consider the one-parameter subgroup 
\begin{eqnarray*}
\rho\colon\G_m&\to&\PGL(2)\times\PGL(2)\subset\Aut(\pr^1\times\pr^1)\\
t&\mapsto&\diag\left((1,t^{-2}), (1, t^{-1})\right).
\end{eqnarray*}
Since $\rho$ fixes the centers of the blowups and $C'$, $\rho$ is a one-parameter 
subgroup of $(X, (1-\beta)C)$. As in \cite[\S 3]{hyperplane}, $\rho$ corresponds to the 
quasi-monomial valuation on 
\[
\A^2_{x_1,x_2}=\pr^1_{z_{10}:z_{11}}\times\pr^1_{z_{20}:z_{21}}\setminus(z_{10}z_{20}=0)
\]
for coordinates $(x_1=z_{11}/z_{10}$, $x_2=z_{21}/z_{20})$ with weights $(2$, $1)$. 
The valuation corresponds to the prime divisor $F$ over $X$ with 
$A_{X, \Delta}(F)/S(F)=1$ in Theorem \ref{caseII_thm} Step 2. This implies the 
K-polystability of $(X, (1-\beta)C)$. 
\end{remark}

\begin{remark}\label{62_rmk}
Assume that $n=2$, $q_1$ lies on a singular fiber of $\eta$ and $q_2$ does not lie on a 
singular fiber of $\eta$. Since $\Aut(X, (1-\beta)C)$ is discrete 
(see \cite[Remark 2.11]{FLSZZ}) and there exists a dreamy prime divisor $F$ over 
$(X, (1-\beta)C)$ such that $A_{X, (1-\beta)C}(F)/S(F)=1$, we can conclude that 
$(X, (1-\beta)C)$ is not K-polystable but K-semistable. 
\end{remark}

\end{document}